\documentclass{article}

\usepackage[english]{babel}
\usepackage[T1]{fontenc}
\usepackage[utf8]{inputenc}
\usepackage{times}
\usepackage{graphicx}
\usepackage[paper=letterpaper, margin=1in]{geometry}
\usepackage[authoryear]{natbib}

\usepackage{pifont} 
\DeclareUnicodeCharacter{00A0}{~}

\usepackage{amssymb,amsmath,amscd,amsfonts,amsthm,bbm,mathrsfs,yhmath}
\usepackage[shortlabels]{enumitem}
\usepackage{hyperref}
\usepackage{appendix}

\usepackage{xcolor}
\usepackage[textwidth=2cm, textsize=footnotesize]{todonotes}
\usepackage[most]{tcolorbox}

\tcbset{
  enhanced, breakable,
  colback=gray!5!white,    
  colframe=blue!15!white,  
  fonttitle=\bfseries,     
  coltitle=black,          
  boxrule=0.8pt,           
  arc=4pt,                 
  left=6pt, right=6pt, top=6pt, bottom=6pt
}

\newtcolorbox{AuthorsBox}[1][]{%
  title=Author's prompt,
  #1
}

\newtcolorbox{ChatGPTBox}[1][]{%
  colframe=red!20!white,
  title=Response of GPT-5-pro,
  #1
}

\newcommand{\KL}{{\mathsf KL}}

\newcommand{\bR}{{\mathbb R}}

\newcommand{\authors}[1]{\textcolor{red}{Author's comment: #1}}

\theoremstyle{definition}
\newtheorem{theorem}{Theorem}

\newtheorem{proposition}[theorem]{Proposition}

\newtheorem{remark}{Remark}

\newtheorem*{example*}{Example}

\title{Accelerating mathematical research with language models: \\ A case study of an interaction with GPT-5-Pro\\ on a convex analysis problem}

\author{Adil Salim}

\begin{document}
\maketitle

\maketitle

\begin{abstract} 
Recent progress in large language models has made them increasingly capable research assistants in mathematics. Yet, as their reasoning abilities improve, evaluating their mathematical competence becomes increasingly challenging. The problems used for assessment must be neither too easy nor too difficult, their performance can no longer be summarized by a single numerical score, and meaningful evaluation requires expert oversight.

In this work, we study an interaction between the author and a large language model in proving a lemma from convex optimization. Specifically, we establish a Taylor expansion for the gradient of the biconjugation operator—that is, the operator obtained by applying the Fenchel transform twice—around a strictly convex function, with assistance from GPT-5-pro, OpenAI’s latest model.

Beyond the mathematical result itself, whose novelty we do not claim with certainty, our main contribution lies in documenting the collaborative reasoning process. GPT-5-pro accelerated our progress by suggesting, relevant research directions and by proving some intermediate results. However, its reasoning still required careful supervision, particularly to correct subtle mistakes. While limited to a single mathematical problem and a single language model, this experiment illustrates both the promise and the current limitations of large language models as mathematical collaborators.
\end{abstract}

\section{Introduction}

\subsection{Mathematical context}
More than two centuries ago,~\citet{monge1781memoire} formulated a question that arises naturally across several areas of science: Given two sets \(U, V \subset \bR^d\) of equal volume, find the optimal volume-preserving map between them, where optimality is measured with respect to a cost function \(c(x,y) \ge 0\). This problem is now known in the mathematical literature as the \emph{optimal transport problem}.  

In a breakthrough result,~\citet{brenier1991polar} showed that when \(c(x,y) = \|x - y\|^2\), the optimal map exists, is unique, and is the gradient of a convex potential. Later,~\citet{gangbo1994elementary} simplified Brenier’s proof. A key ingredient of this simplification is the following simple convex analysis lemma (Lemma~2.4 therein):
\begin{equation}
\label{eq:wilfrid}
(\phi^{*} + t\,h)^{*}(x) = \phi(x) - t\,h(\nabla \phi(x)) + o(t),
\end{equation}
for almost every (a.e.) \(x \in \bR^d\), where \(\phi : \bR^d \to \bR\) is convex, \(h\) is a (possibly nonconvex) test function, and \({}^{*}\) denotes the Fenchel transform\footnote{The expansion is exact at $t=0$ because $(\phi^{*})^{*} = \phi$.}.  
This lemma was later extended to \(c\)-transforms by~\citet{gangbo1996geometry}, to solve the optimal transport problem for general costs of the form \(c(x,y) = h(\|x - y\|)\), with \(h\) strictly convex.

As part of a research project on optimal transport, we aimed to prove the following variant of~\ref{eq:wilfrid}, assuming $\phi$ is strictly convex:
\begin{equation}
\label{eq:us}
\nabla (\phi + t h)^{**}(x) = \nabla \phi(x) + t \nabla h(x) + o(t).
\end{equation}
Without additional assumptions, we could not rigorously derive~\eqref{eq:us} by applying~\eqref{eq:wilfrid} twice in a straightforward way, nor did we find it stated in the convex analysis literature. However, given its apparent simplicity, we do not claim the novelty of~\eqref{eq:us} with certainty.

In this paper, we present a proof of Equation~\eqref{eq:us}. In particular, we show that the remainder term $o(t)$ is equal to zero if $|t|$ is small enough (depending on $x$). But, beyond the simple mathematical result, our main contribution is to document the interaction with GPT-5-pro that led to the proof of~\eqref{eq:us}. 

\subsection{Social context}

Language models such as Claude, Gemini and GPT are machine learning systems trained on vast amounts of text data to answer prompts without human supervision. Their performance is typically measured by how much they can help---or potentially harm---the human user. 

Although they can be tested on nearly any domain of human knowledge, they are often evaluated on mathematics, as it is viewed as a direct measure of reasoning or "intelligence".  

With the recent increase in the reasoning abilities of language models, evaluating them in mathematics has become increasingly challenging. Many benchmark problems are now too easy, while most research questions remain too difficult. An alternative to automated benchmarks is to evaluate language models as research assistants. That requires setting up realistic research collaborations---a process that is inherently qualitative and non-scalable, since it depends on human experts, extended interactions, and subjective judgment.

Since the release of GPT-5-pro by OpenAI in August 2025, several mathematicians have documented their use of GPT-5-pro to solve math problems. In a post on X, S\'ebastien Bubeck showed how GPT-5-pro could improve a bound on the step size of an optimization algorithm. Then,~\citet{diez2025mathematical} used GPT-5-pro to compute the convergence speed of a Central Limit Theorem. Later on,~\cite{feldman2025g} submitted five combinatorial optimization open questions to GPT-5-pro, and the model could solve three of them. In a post on X, Paata Ivanisvili showed that GPT-5-pro could provide a counter example to a conjecture made in Fall 2013, in the context of a Simons Institute program on real analysis. Finally, Ernest Ryu recently announced on X that he could prove the pointwise convergence of the Nesterov ODE with help of GPT-5-pro.

\subsection{Contribution}
Beyond~\citet{diez2025mathematical}, we have limited available evaluation data of frontier language models as research collaborators in mathematics. The goal of this study is to add a data point to this evaluation. 

To this end, we focus on proving Equation~\eqref{eq:us} with help of GPT-5-pro, and we document our interaction with the model. We believe that proving Equation~\eqref{eq:us} is an interesting math problem that is not trivial, but simple enough to be discussed with a language model and to be understood by a large number of mathematicians.

Beyond the mathematical result~\eqref{eq:us}, this study gives an example of a productive use of a language model for research. GPT-5-pro accelerated our progress by suggesting relevant research directions such as the key conjecture, and by proving some lemmas. But its reasoning still required careful supervision, particularly to correct mistakes.

\subsection{Outline}
This paper is organized as follows. In Section~\ref{sec:math}, we state and prove the main mathematical result~\eqref{eq:us}. The text {\color{blue} in blue} indicates the ideas that mostly came from GPT-5-pro and the text in black indicates the ideas that mostly came from us. In Section~\ref{sec:summary}, we summarize the interaction we had with the language model. We conclude in Section~\ref{sec:ccl}. The full chat can be found in Appendix, with our comments on some of the model's mistakes.  

\section{Main mathematical result}
\label{sec:math}
Before presenting our main result, we specify some notations. For every function $f: \bR^d \to (-\infty,+\infty]$, we denote by $f^{*}$ its Fenchel transform
\begin{equation}
f^{*}(y) := \sup_{x \in \bR^d} \langle x,y \rangle - f(x).
\end{equation}
The biconjugate $(f^{*})^{*}$ of $f$ is denoted $f^{**}$. For every $x \in \bR^d$, the subdifferential of $f$ at $x$, denoted $\partial f(x)$, is the subset of $\bR^d$ defined by
\begin{equation}
p \in \partial f(x) \Longleftrightarrow \forall y \in \bR^d, \quad  f(y) \geq f(x) + \langle p , y - x \rangle.
\end{equation}
We now state our main result. 
\begin{proposition}
\label{prop:biconj}
{\color{blue} Let $x \in \bR^d$ and $\phi:\bR^d\to\bR$ be convex and twice differentiable at $x$. Assume
$\nabla^2\phi(x)$ symmetric positive definite, and set
\[
\lambda := \lambda_{\min}\big(\nabla^2\phi(x)\big) \;>\; 0.
\]
Let $h \in C_c^2(\bR^d)$ with compact support, and let $L := \|\nabla^2 h\|_{\infty}$ and $M := \| h\|_{\infty}$.}
Define $f:=\phi + t h$ and set
\[
t_x := \min \left(\frac{\lambda}{4L}, \frac{\lambda \delta^2}{64 M}\right) > 0,\]
where $\delta > 0$ is a constant depending on $x$ defined in the proof.
{\color{blue} Then, for every $t \in (-t_x,t_x)$,
\begin{equation}
\label{eq:fct}
(\phi+th)^{**}(x) \;=\; (\phi+th)(x),
\end{equation}
and 
\begin{equation}
\label{eq:grad}
\nabla(\phi+th)^{**}(x) = \nabla \phi(x)+t \nabla h(x).
\end{equation}}
\end{proposition}

\begin{remark}
The subtlety of Proposition~\ref{prop:biconj} lies in its locality. We do not know \textit{a priori} if $(\phi + t h)$ is convex for $|t|$ small enough, nor we know if $\nabla^2 \phi$ (well defined a.e. because $\phi$ is convex) is positive definite in a neighborhood of $x$\footnote{Even if we were assuming $\phi$ strictly convex, we would not know \textit{a priori} that the eigenvalues of $\nabla^2 \phi$ are uniformly bounded below by a positive constant in a neighborhood of $x$.}. Hence, Equation~\eqref{eq:fct}, which is valid at $x$, might fail to hold around $x$, and Equation~\eqref{eq:grad} cannot be obtained directly by differentiating Equation~\ref{eq:fct}.
\end{remark}    

\begin{proof}

{\color{blue}
By Taylor at $x$, there exists a function
$r : \bR^d \to \bR$ with $r(u)\to 0$ as $u\to 0$ such that,
\begin{equation}\label{eq:psi-taylor-1}
\phi(x+ u)
=
\phi(x) + \langle \nabla\phi(x),u\rangle
+ \frac12\langle u, \nabla^2\phi(x) u\rangle
+ r(u)\,\|u\|^2 .
\end{equation}
Using $\langle u, \nabla^2\phi(x) u\rangle \geq \lambda \|u\|^2$, 
\begin{equation}\label{eq:psi-taylor}
\phi(x+ u)
\;\ge\;
\phi(x) + \langle \nabla\phi(x),u\rangle
+ \left(\frac{\lambda}{2} + r(u)\right)\,\|u\|^2.
\end{equation}
There exists $\delta > 0$ such that $|r(u)| < \lambda / 4$ if $\|u\| < \delta$. Therefore, for every $u$ such that $\|u\| < \delta$,
\begin{equation}\label{eq:psi-quad}
\phi(x+ u)
\;\ge\;
\phi(x) + \langle \nabla\phi(x),u\rangle
+ \frac{\lambda}{4}\,\|u\|^2.
\end{equation}
Besides, for every $t \in \bR$, $\nabla (th)$ is $|t|L$–Lipschitz. Therefore,
\begin{equation}\label{eq:h-smooth}
(th)(x+ u)
\;\ge\;
(th)(x) + \langle \nabla (th)(x),u\rangle
- \frac{|t|L}{2}\,\|u\|^2 .
\end{equation}

\paragraph{1) On the ball of radius $\delta$ centered at $x$.}
Adding \eqref{eq:psi-quad} and \eqref{eq:h-smooth} gives, for
$\|u\|\le \delta$,
\begin{equation}
\label{eq:quad}
f(x+ u)
\;\ge\;
f(x) + \langle \nabla f(x),u\rangle
+ \Big(\tfrac{\lambda}{4}-\tfrac{|t|L}{2}\Big)\|u\|^2 .
\end{equation}
Taking $|t|<\lambda/4L$, the bracket is at least $\lambda/8$,
\begin{equation}\label{eq:on-ball}
f(x+u) \;\ge\; f(x)+\langle \nabla f(x),u\rangle + \frac{\lambda}{8} \|u\|^2
\qquad \text{for all}\ \|u\|\le \delta .
\end{equation}
}
{\color{blue}\paragraph{2) On the sphere of radius $\delta$ centered at $x$.}

If $\|u\|=\delta$, the bound~\eqref{eq:quad} yields a positive margin
\begin{equation}
f(x+ u)
\;\ge\;
f(x) + \langle \nabla f(x),u\rangle
+ \frac{\lambda}{8}\delta^2 .
\end{equation}}
We write it in terms of $\phi$,
\begin{equation}
\phi(x+ u)
\;\ge\;
\phi(x) + \langle \nabla f(x),u\rangle
+ \frac{\lambda}{8}\delta^2 + t(h(x) - h(x+u)).
\end{equation}
Since $|t(h(x) - h(x+u))| \leq 2 |t| M$, if $|t|<\lambda \delta^2/(32M)$ we have $t(h(x) - h(x+u)) > -\lambda \delta^2 / 16$. Therefore, for all $u \in \bR^d$ such that $\|u\|=\delta$,
\begin{equation}\label{eq:sphere-margin-f}
\phi(x+u) \;\ge\; \phi(x) + \langle \nabla f(x),u\rangle + \frac{\lambda}{16}\delta^2 .
\end{equation}
The r.h.s. defines an affine map that we denote $L(u) := \phi(x) + \langle \nabla f(x),u\rangle + \frac{\lambda}{16}\delta^2$.
\paragraph{4) Outside the ball of radius $\delta$ centered at $x$.}
We have shown that $\phi(x+u) \geq L(u)$ if $\|u\| = \delta$. {\color{blue} We first show that this is still true if $\|u\| > \delta$.} Assume by contradiction that there exists $\bar{u}$ such that $\|\bar{u}\| > \delta$ and $\phi(x+\bar{u}) < L(\bar{u})$. Note that we obviously have $\phi(x) = \phi(x+0) < L(0)$. Let $g(u) = \phi(x+u) - L(u)$. Then $g$ is convex, $g(0) < 0$ and $g(\bar{u}) < 0$. By convexity, $g(\alpha \bar{u}) < 0$ for every $\alpha \in [0,1]$. Noting that $\alpha = \delta / \|\bar{u}\| \in [0,1]$, we have $\phi(x+\alpha\bar{u}) < L(\alpha\bar{u})$. However, $\|\alpha\bar{u}\| = \delta$ which contradicts~\eqref{eq:sphere-margin-f}. {\color{blue}Therefore, for all $u$ s.t. $\|u\| \geq \delta$, $\phi(x + u) \geq L(u)$, \textit{i.e.},
\begin{equation}
\phi(x+u) \;\ge\; \phi(x) + \langle \nabla f(x),u\rangle + \frac{\lambda}{16}\delta^2.
\end{equation}}
Since $|t(h(x) - h(x+u))| \leq 2 |t| M$, if $|t|<\lambda \delta^2/(64M)$ we have $$t(h(x+u) - h(x)) > -\frac{\lambda \delta^2}{32}.$$ Adding the last two inequalities, we obtain
\begin{equation}
\label{eq:outside-ball}
f(x+u) \;\ge\; f(x) + \langle \nabla f(x),u\rangle + \frac{\lambda}{32}\delta^2, \qquad \text{for all}\ \|u\|\ge \delta 
\end{equation}
\smallskip
\paragraph{Conclusion.}
Considering both cases $\|u\| \leq \delta$ or $\|u\| \geq \delta$, we have using~\eqref{eq:on-ball} and~\eqref{eq:outside-ball}
\begin{equation}
\label{eq:last1}
f(x + u) \;\ge\; f(x) + \langle \nabla f(x),u\rangle \qquad \text{for all}\ u \in \bR^d, 
\end{equation}
and the inequality is strict if $u \neq 0$. Equivalently, 
\begin{equation}
\label{eq:last}
f(y) \;\ge\; f(x) + \langle \nabla f(x),y - x\rangle \qquad \text{for all}\ y \in \bR^d. 
\end{equation}
Consider the subset of $\bR$ defined by 
$$
E := \{a(x), \quad a: \bR^d \to \bR \text{ affine function s.t. } \forall y \in \bR^d, a(y) \leq f(y) \}. $$
Taking $a(y) = f(x) + \langle \nabla f(x),y - x\rangle$, we see using~\eqref{eq:last} that $f(x) \in E$. Besides, for all $\alpha \in E$, $\alpha \leq f(x)$. Therefore $f(x) = \sup E$. Using the characterization of $f^{**}$ as the affine envelope of $f$ (see~\cite[Section 12]{Roc70}), $f^{**}(x) = \sup E$. We conclude that $f^{**}(x) = f(x)$, \textit{i.e.}, that Equation~\eqref{eq:fct} holds. 

\paragraph{Gradient.}
Since $f^{**}$ is convex and $f^{**}(y) \leq f(y) < \infty$ for all $y \in \bR^d$, $\partial f^{**}(y) \neq \emptyset$ using~\cite[Theorem 23.4]{Roc70}. 
{\color{blue} Let $z \in \partial f^{**}(x)$, \textit{i.e.},
\begin{equation}
f^{**}(x + u) \;\ge\; f^{**}(x) + \langle z,u\rangle \qquad \text{for all}\ u \in \bR^d. 
\end{equation}
Since $f^{**}(x + u) \leq f(x + u)$ and $f^{**}(x) = f(x)$,
\begin{equation}
f(x + u) \;\ge\; f(x) + \langle z,u\rangle \qquad \text{for all}\ u \in \bR^d. 
\end{equation}
Replacing $u$ by $\varepsilon u$, $\varepsilon > 0, \|u\| = 1$, we rewrite the last inequality as
\begin{equation}
\frac{f(x + \varepsilon u) - f(x)}{\varepsilon} \;\ge\; \langle z,u\rangle \qquad \text{for all}\ \|u\| = 1. 
\end{equation}
Letting $\varepsilon \to 0$ and using that $f$ is differentiable at $x$,
\begin{equation}
\langle \nabla f(x),u\rangle \;\ge\; \langle z,u\rangle \qquad \text{for all}\ \|u\| = 1, 
\end{equation}
which classically implies $z = \nabla f(x)$}. In other words, $\partial f^{**}(x) = \{\nabla f(x)\}$. Using~\cite[Theorem 25.1]{Roc70}, $f^{**}$ is differentiable at $x$ and $\nabla f^{**}(x) = \nabla f(x)$. 
\end{proof}

{\color{blue}
\begin{proposition}
Let $\phi:\bR^d\to\bR$ be strictly convex and $h \in C_c^2(\bR^d)$. Let $F(x,t) := \nabla(\phi+th)^{**}(x)$. Then, for a.e. $x \in \bR^d$, $F$ is derivable at $t = 0$ and 
\[
\partial_2 F(x,0) = \nabla h(x).
\]
\end{proposition}
\begin{proof}
The Alexandrov Hessian $\nabla^2 \phi(x)$ exists and is symmetric positive definite for a.e. $x \in \bR^d$ (see~\cite[Theorem 5.5.4]{ambrosio2008gradient}). Using Proposition~\ref{prop:biconj},
\[
\frac{\nabla(\phi+th)^{**}(x) - \nabla \phi(x)}{t} = \nabla h(x),
\]
if $|t| < t_x$. Recalling that $\phi = \phi^{**}$ and letting $t \to 0$ we get the conclusion.
\end{proof}}
The last proposition is equivalent to Equation~\eqref{eq:us}.
\section{Summary of the discussion}
\label{sec:summary}
We now summarize the chat with GPT-5-pro presented in the Appendix.

The broad context of the discussion was an optimal transport problem: compute a velocity field at $t = 0$ for the curve $t \mapsto \nu_t := \nabla(\phi + t h)^{**} \# \mu$, where $\mu$ is a given probability measure and $\#$ the pushforward operation. To this end, we wanted to establish~\eqref{eq:us} for a.e. $x \in \bR^d$.

When we asked GPT-5-pro to prove Proposition~\ref{prop:biconj} directly, the model mostly provided incorrect proofs. It correctly established a quadratic lower bound on $\phi$ around $x$ first, but then wrote that this lower bound implies strong convexity of $\phi$ around $x$. However, we made no inference about the continuity of the Alexandrov Hessian $\nabla^2 \phi$, and so we do not know that around $x$, the eigenvalues of $\nabla^2 \phi$ are uniformly bounded below by a positive constant.

After pointing out this issue, GPT-5-pro proposed a slightly different approach relying on showing that (i) the Taylor expansion of $(\phi + th)^{**}(x)$ is valid at $x$ only, and not necessarily on a neighborhood of $x$ (ii) the $o(t)$ remainder term in the Taylor expansion is actually zero for $|t|$ small enough. Even if GPT-5-pro did not manage to prove these two items immediately, (i) and (ii) constitute the key conjecture. Next, we suggested the language model a technique to prove (i) and (ii). More precisely, we asked GPT-5-pro to show that $(\phi + th)$ is above its tangent at $x$ (see the "Conclusion" paragraph in the proof of Proposition~\ref{prop:biconj}). 

GPT-5-pro correctly explained why $(\phi + th)$ is above its tangent at $x$ on a ball of radius $r>0$ centered at $x$. It correctly suggested that this result can be extended outside the ball (see in the Appendix references to the "boundary gap"), but it made mistakes when trying to prove that result, repeatedly using that $-\|y-x\| > -r$ for $y$ outside the ball. We completed the proof of Equation~\eqref{eq:fct} ourselves.

Because Equation~\eqref{eq:fct} is, in general, not valid in a neighborhood of $x$, one cannot differentiate it to directly obtain~\eqref{eq:grad}. Instead, GPT-5-pro wrote in one shot a proof of~\eqref{eq:grad} relying on a convex analysis argument. 

In summary, GPT-5-pro proposed several interesting directions. Although it did not produce a fully correct proof, it identified the key conjecture, established several intermediate results, and provided enough insight for us to complete the proof efficiently.

\section{Conclusion}
\label{sec:ccl}

We investigated the collaboration between the author and GPT-5-pro in establishing Equation~\eqref{eq:us}. The language model significantly accelerated the research process, although human supervision remained necessary.

This study presents a qualitative case limited to a single task and a single language model. While it highlights the promise of language models as research assistants, developing systematic and scalable frameworks to assess such collaborations remains an important direction for future work.

\section{Acknowledgment}

We thank Wilfrid Gangbo and Dimitri Shlyakhtenko for insightful discussions. Wilfrid Gangbo reviewed the proof of Proposition~\ref{prop:biconj} and made several comments on this study, some of which were included. Part of this work was done while the author was visiting the Institute for Pure and Applied Mathematics (IPAM) at UCLA.

 \newcommand{\noop}[1]{} \def\cprime{$'$} \def\cdprime{$''$} \def\cprime{$'$}

\newpage

\appendix

\section{Interaction with ChatGPT-5-pro}

The following chat has been minimally edited for clarity. We comment in red on some of the mistakes of the language model. The broad topic of the chat was an optimal transport problem. While we extracted the part of the chat that is relevant to our Equation~\eqref{eq:us}, the reader might see references to a broader math problem. We indicate the text that was erased using [...].

\begin{AuthorsBox}

OK. but I am not sure of the moreau envelope smoothing showing that the $g_\epsilon$ converges to zero. 

Nevermind. I think we could greatly simplify the proof if we use $(\phi + t h )^{**}$ instead. Because formally the Taylor expansion should give $\phi + t h + o(t)$, and so the $A$ matrix will not appear after taking the gradient.  

\begin{enumerate}
    \item Find a reference for that. If you dont find, try to prove it
    \item Make a list of all elements of the proof that we need to adapt. I think the first item is to take the gradient in this taylor expansion
    \item Prove all items in the list
\end{enumerate}

Let's go!

\end{AuthorsBox}

\medskip

\begin{ChatGPTBox}
Love the idea—switching to the path
\[
\phi_t := (\phi+t\,h)^{**}
\qquad (t\ \text{small})
\]
so that we always stay inside the class “gradient of a convex potential” and, morally, the first–order perturbation in $t$ is just $h$. Below I (i) point to the right references, (ii) list precisely what must be adapted, and (iii) prove each item at the level we need. I’ll be very explicit about where we rely on standard theorems (with citations) and where we give short self-contained arguments.

\medskip
\noindent\rule{\textwidth}{0.4pt}

\subsection*{1) References for the first-order expansion of $(\phi+t h)^{**}$}

The right language is \textbf{epi-differentiability} and \textbf{tilt-stability} in variational analysis.

\begin{itemize}
    \item \textbf{Rockafellar–Wets, Variational Analysis (1998)}:
    \begin{itemize}
        \item Epi-convergence and epi-derivatives (Ch. 7, 13).  
        \item Stability at points with \textbf{single subgradient} (i.e., where the convex $\phi$ is differentiable, equivalently $\partial\phi(x)=\{\nabla\phi(x)\}$).  
        \item \textbf{Key fact:} biconjugation $f\mapsto f^{**}$ is epi-continuous and preserves first-order epi-derivatives; at points where the active subgradient is unique, the first epi-derivative of $t\mapsto (f+t h)^{**}$ in direction $h$ is exactly $h$ (no convexification needed).
        \item See especially the “second-order convex analysis” notes by Rockafellar for quick access: first-order expansions via epi-derivatives and the proto-derivative of the subdifferential mapping (OU lecture notes “Second-Order Convex Analysis”).
    \end{itemize}

    \item \textbf{Poliquin–Rockafellar, “A calculus of epi-derivatives applicable to nonconvex functions” (1996)}: chain rules for epi-derivatives of composites and for perturbations $f+t h$ with smooth $h$. In particular, when the outer operation is biconjugation (closed convex hull), the first epi-derivative passes through the operation, and if $\partial\phi(x)$ is a singleton, the epi-derivative at $\phi$ in direction $h$ is just $h$ (pointwise at $x$).
\end{itemize}

\authors{Here, the model seems to be confused between Taylor expansion of $(\phi + th)^{**}(x)$, where $h$ is a function, and Taylor expansion of $\phi^{**}(x + th)$, where $h$ is a vector. The "takeaway" below do not seem to follow from the references cited above.}

\textbf{Takeaway:} At $\mu_0$-a.e.\ $x$ where $\phi$ is twice differentiable with $\nabla^2\phi(x)\succ0$ (true by Alexandrov, and the positive definiteness is the transport one-to-one a.e.\ fact you already used), the map
\[
t\longmapsto (\phi+t h)^{**}(x)
\]
admits the first–order expansion
\[
(\phi+t h)^{**}(x)= \phi(x) + t\,h(x) + o(t)
\quad\text{and}\quad
\nabla (\phi+t h)^{**}(x)= \nabla\phi(x) + t\,\nabla h(x) + o(t),
\]
for every $h\in C^2_c(\mathbb R^d)$.

If you prefer an “elementary” proof in our setting (no heavy epi-derivative jargon), see Item (A) below: we give a local argument that near any point where $\nabla^2\phi(x)\succ0$, the convexification has \textbf{no effect to first order}—indeed, for small $|t|$ the biconjugate coincides with $\phi+t h$ on a small ball around $x$, so the gradient expansion becomes $\nabla\phi(x)+t\nabla h(x)+o(t)$.

\authors{The statement "the biconjugate coincides with $\phi+t h$ on a small ball around $x$" is incorrect, but we will see that it can be corrected. [...]}

\subsection*{3) Proofs of the items}

I’ll first give an \textbf{elementary local argument} for the first-order expansion (Item 1), then indicate how the epi-derivative reference encapsulates it. The remaining items then follow quickly.

\subsubsection*{(A) Local first-order expansion of $(\phi+t h)^{**}$ near points with $\nabla^2\phi(x)\succ0$}

Fix $x$ where $\phi$ is twice differentiable with $\nabla^2\phi(x)\succ0$. Let $\lambda_{\min}>0$ be the smallest eigenvalue of $\nabla^2\phi(x)$. By Alexandrov’s second-order expansion, for small $r>0$,
\[
\phi(x+u)= \phi(x)+\langle \nabla\phi(x),u\rangle + \tfrac12\langle \nabla^2\phi(x)u,u\rangle + o(\|u\|^2)
\quad\text{for }\|u\|\le r.
\]
Hence, shrinking $r$ if needed, there is $m\in(0,\lambda_{\min})$ such that
\[
\phi(x+u)\ \ge\ \phi(x)+\langle \nabla\phi(x),u\rangle + \tfrac{m}{2}\|u\|^2
\quad\text{for }\|u\|\le r.
\tag{1}
\]
Let $M:=\sup_{B(x,r)}\|\nabla^2 h\|$ (finite since $h\in C^2_c$). Choose $|t|\le t_0:=m/(2M)$. Then on $B(x,r)$,
\[
(\phi+t h)(x+u)\ \ge\ \phi(x)+\langle \nabla\phi(x),u\rangle + \tfrac{m}{2}\|u\|^2
\;+\; t\big(h(x)+\langle \nabla h(x),u\rangle\big) - |t|\,\tfrac{M}{2}\|u\|^2,
\]
which gives the \textbf{local strong convexity} estimate
\[
(\phi+t h)(x+u)\ \ge\ (\phi+t h)(x)+\langle \nabla\phi(x)+t\nabla h(x),u\rangle
+ \tfrac{m-|t|M}{2}\|u\|^2
\ \ge\ (\cdots) + \tfrac{m}{4}\|u\|^2
\tag{2}
\]
for $|t|\le t_0$. In particular, $\phi+t h$ is $m/4$-strongly convex on $B(x,r)$; thus it is $C^{1,1}$ there and
\[
\nabla(\phi+t h)(x)= \nabla\phi(x)+t\,\nabla h(x).
\tag{3}
\]
\authors{Inequality (2) proves that $\phi+th$ is above a quadratic approximation on the ball, but this does not imply strong convexity on the ball. The end of the proof is wrong because it relies on this strong convexity on the ball. [...]}

\end{ChatGPTBox}  \begin{AuthorsBox}  You said: "General fact (finite dimension): if a scalar function 
$f:\mathbb{R}^d\to\mathbb{R}$ is twice Fréchet differentiable at a point 
$y_0$
 with second derivative 
$H$, then its gradient is strictly (Fréchet) differentiable at 
$y_0$
 with derivative 
$H$. Equivalently, (SD) holds with 
$B=H$. This is standard; see, e.g., Rockafellar–Wets, \textit{Variational Analysis}, §13.E (second-order expansion 
$\Rightarrow$ strict differentiability of the gradient)."

But it seems that Rockafellar–Wets gives strict differentiability of the function (not its gradient), correct?

\authors{Here, we challenge a statement made by the model in previous messages that are not reported here.}

\medskip

\end{AuthorsBox}  \begin{ChatGPTBox} You’re right to question that citation. 

\authors{[...]}

\subsubsection*{Route 1 (what we’re now doing): switch to the \textbf{primal path} $\phi_t=(\phi+t h)^{**}$}

For convex $\phi$ that is twice differentiable at $x$ with \textbf{positive-definite} Hessian $\nabla^2\phi(x)$, one has a \textbf{local strong convexity} around $x$. A simple quadratic upper/lower sandwich shows that, for $|t|$ small and $z$ near $x$, the convexification is \textbf{inactive}:
\[
(\phi+t h)^{**}=\phi+t h\quad\text{on a small ball around }x.
\]
\authors{The last equation is close to the key conjecture, but the statement is true at $x$ only and not in a neighborhood of $x$.}

Consequently,
\[
\nabla(\phi+t h)^{**}(x)=\nabla\phi(x)+t\,\nabla h(x) \quad\text{and}\quad
\frac{d}{dt}\Big|_{t=0}\nabla(\phi+t h)^{**}(x)=\nabla h(x)
\]
without any need for strict differentiability of $\nabla\phi^*$.  
\textit{This is the elementary local argument we already wrote out in detail (local strong convexity + “convexification doesn’t move the graph to first order”).}

\authors{[...]}

\subsection*{Bottom line for your specific question}

\begin{itemize}
    \item \textbf{Yes}, your suspicion is correct: RW §13.E does not directly give strict differentiability of the \textbf{gradient} from second-order differentiability of $f$. It gives the \textbf{second-order expansion of $f$} (equivalently, Fréchet differentiability of $\nabla f$ at a point), which is weaker than the pairwise estimate we labeled (SD).
    \item In our proof we no longer need that stronger (SD) fact, because with the \textbf{primal variation} $(\phi+t h)^{**}$ the first-order formula
    \[
    \nabla(\phi+t h)^{**}(x)=\nabla\phi(x)+t\,\nabla h(x)+o(t)
    \]
    follows from a local “convexification is inactive” lemma—an elementary argument that uses only convexity plus the positive-definite Hessian at $x$. This bypasses the delicate “strict differentiability of $\nabla\phi^*$” issue altogether.
\end{itemize}

\authors{[...]}

\medskip

\end{ChatGPTBox}

\begin{AuthorsBox}   Add more details for Route 1

\medskip

\end{AuthorsBox} \begin{ChatGPTBox}       Great—here is \textbf{Route 1} written out in full detail, with every “why” spelled out. The core idea is to vary along
\[
\phi_t \;:=\; (\phi + t\,h)^{**}\quad (t\ \text{small},\ h\in C_c^2(\mathbb R^d)),
\]
so we remain in the class of convex potentials, and to show that \textbf{to first order} the convexification is \textit{inactive} near almost every point $x$ where $\phi$ has a positive–definite Hessian. This yields
\[
\nabla\phi_t(x) \;=\; \nabla\phi(x) + t\,\nabla h(x) + o(t)
\quad(\mu_0\text{-a.e.}),
\]
i.e. the velocity field is simply $v_0(y)=\nabla h(\nabla\phi^*(y))$. We then use \textbf{OU smoothing after the pushforward} to compute the entropy derivative classically, and conclude as before.

Below: (0) the “good” set and why $\nabla^2\phi\succ0$ $\mu_0$-a.e.; (1) the local “convexification is inactive” lemma; (2) the gradient expansion and velocity; (3) why we still smooth, and why \textbf{after} the pushforward; (4) the one-sided derivatives and the limiting Stein identity; (5) the completion to the \textbf{standard} Stein identity.

\medskip
\noindent\rule{\textwidth}{0.4pt}

\subsection*{0) The “good” set and positive–definite Hessian a.e.}

\begin{itemize}
    \item $\mu_0\ll \mathcal L^d$ and $\nu_0:=\nabla\phi \#\mu_0\ll\mathcal L^d$ (since $\KL(\nu_0\mid\pi)\le \KL(\mu_0\mid\pi)<\infty$).  
    \item By Alexandrov, $\phi$ and $\phi^*$ are twice differentiable a.e.\ w.r.t.\ Lebesgue; since $\mu_0,\nu_0$ are a.c., we may pick a set $G\subset\mathbb R^d$ of full $\mu_0$-measure such that for all $x\in G$:
    \begin{itemize}
        \item $\nabla\phi(x)$ exists; set $y_0=\nabla\phi(x)$.
        \item $\nabla\phi^*(y_0)$ exists and equals $x$ (subgradient reciprocity + $\nu_0$-a.e.\ differentiability).
        \item $\nabla^2\phi(x)$ and $\nabla^2\phi^*(y_0)$ both exist. Differentiating $\nabla\phi^*(\nabla\phi(x))=x$ at such $x$ gives
        \[
        \nabla^2\phi^*(y_0)\,\nabla^2\phi(x)=I_d.
        \]
        Hence $\nabla^2\phi(x)\succ0$, $\nabla^2\phi^*(y_0)\succ0$, and they are inverses.
    \end{itemize}
\end{itemize}

We work pointwise on this full $\mu_0$-measure set.

\medskip
\noindent\rule{\textwidth}{0.4pt}

\subsection*{1) “Convexification is inactive” near $x$}

\textbf{Lemma (local inactivity of biconjugation).}  
Fix $x\in G$, let $H:=\nabla^2\phi(x)\succ0$ and $h\in C_c^2$. There exist $r>0$ and $t_0>0$ such that for all $|t|\le t_0$,
\[
(\phi+t h)^{**}(z)\;=\;\phi(z)+t\,h(z)\quad\text{for all }z\in B(x,r),
\]
and consequently
\[
\nabla(\phi+t h)^{**}(x)\;=\;\nabla\phi(x)+t\,\nabla h(x).
\]

\textbf{Proof (fully elementary).}

\begin{enumerate}
    \item \textit{Local quadratic lower bound for $\phi$.}  
    By Alexandrov at $x$, for small $r_1>0$,
    \[
    \phi(x+u)\ \ge\ \phi(x)+\langle \nabla\phi(x),u\rangle+\tfrac12\langle H u,u\rangle - \varepsilon_1\|u\|^2,\qquad \|u\|\le r_1,
    \]
    with $\varepsilon_1\in(0,\tfrac12\lambda_{\min}(H))$. Set $m:=\lambda_{\min}(H)-2\varepsilon_1>0$. Shrinking $r_1$ if needed,
    \[
    \phi(x+u)\ \ge\ \phi(x)+\langle \nabla\phi(x),u\rangle + \tfrac{m}{2}\|u\|^2,\qquad \|u\|\le r_1. \tag{1}
    \]
    
    \item \textit{Control $h$.}  
    Let $M:=\sup_{B(x,r_1)}\|\nabla^2 h\|<\infty$. For all $|t|\le t_0:=m/(2M)$,
    \[
    h(x+u)\ \ge\ h(x)+\langle \nabla h(x),u\rangle - \tfrac{M}{2}\|u\|^2.
    \]
    
    \item \textit{Strong convexity of $\phi+t h$ on $B(x,r_1)$.}  
    Combine (1) with the previous bound: for $|t|\le t_0$ and $\|u\|\le r_1$,
    \[
    (\phi+t h)(x+u)\ \ge\ (\phi+t h)(x)+\langle \nabla\phi(x)+t\nabla h(x),u\rangle + \tfrac{m-|t|M}{2}\|u\|^2
    \ \ge\ (\cdots)+\tfrac{m}{4}\|u\|^2. \tag{2}
    \]
    Thus $\phi+t h$ is $C^{1,1}$ and $\frac{m}{4}$-strongly convex on $B(x,r_1)$, and in particular
    \[
    \nabla(\phi+t h)(x)\;=\;\nabla\phi(x)+t\,\nabla h(x). \tag{3}
    \]
    
    \item \textit{Convexification does nothing near $x$.}  
    
    \authors{The argument below relies again on strong convexity on the ball, which is not true. [...]}
    
    Strong convexity on a ball implies a Jensen–variance inequality: for any convex combination $z=\sum\lambda_i z_i$ with $z_i\in B(x,r_1)$,
    \[
    \sum \lambda_i(\phi+t h)(z_i) \ \ge\ (\phi+t h)(z) + c\,\sum\lambda_i\|z_i-z\|^2,\qquad c:=\tfrac{m}{8}>0. \tag{4}
    \]
    Choose $r\in(0,r_1)$ small enough so that any convex combination with barycenter $z\in B(x,r)$ that places any mass outside $B(x,r_1)$ has $\sum\lambda_i\|z_i-z\|^2\ge \delta>0$ (geometric fact: points $\ge r_1$ away force positive variance when the barycenter is within $r$).  
    Then (4) shows that for $z\in B(x,r)$ the \textbf{infimum over all convex decompositions} occurs with all $z_i\in B(x,r_1)$. But on $B(x,r_1)$ the function $\phi+t h$ is already convex (indeed strongly convex), so its convex envelope coincides with itself. Equivalently, the biconjugate $(\phi+t h)^{**}$ equals $\phi+t h$ on $B(x,r)$. This proves the first assertion, and the gradient identity follows by differentiating $\phi+t h$ at $x$ (using (3)). \qed
\end{enumerate}

\textit{Remarks.}  
\begin{itemize}
    \item The constants $r,t_0$ depend on the point $x$, but the \textbf{conclusion holds for $\mu_0$-a.e.\ $x$}. For the derivative at $t=0$ this is exactly what we need.  
    \item The proof uses only convexity, Alexandrov’s second-order expansion at $x$, and a tiny Taylor bound for $h$; \textbf{no} assumption on $\phi^*$ is needed.
\end{itemize}

\medskip
\noindent\rule{\textwidth}{0.4pt}

\subsection*{2) Gradient expansion and the velocity}

From the lemma, for $\mu_0$-a.e.\ $x$ there exists $t_x>0$ such that, for all $|t|\le t_x$,
\[
\nabla\phi_t(x)\;=\;\nabla\phi(x)+t\,\nabla h(x).
\]
Thus the \textbf{Lagrangian velocity} at $t=0$ is $\dot y_0(x):=\partial_t|_{t=0}\nabla\phi_t(x)=\nabla h(x)$. Passing to Eulerian form with $y=\nabla\phi(x)$ and $x=\nabla\phi^*(y)$, we get the \textbf{Eulerian velocity}
\[
v_0(y)\;=\;\dot y_0\!\big(\nabla\phi^*(y)\big)\;=\;\nabla h\!\big(\nabla\phi^*(y)\big).
\]
Since $h\in C_c^2$, $\nabla h$ is bounded and compactly supported; hence $v_0$ is bounded and compactly supported as well.

\authors{[...]}

\end{ChatGPTBox} \begin{AuthorsBox}   Ok reexplain why the local quadratic lower bound implies equality between the function and its biconjugate, locally

\medskip

\end{AuthorsBox} \begin{ChatGPTBox}       Great question. The right statement to prove is:

\textbf{Claim.} Fix $x$ where $\phi$ is twice Fréchet differentiable with
$H:=\nabla^2\phi(x)\succ0$. For any $h\in C_c^2(\mathbb R^d)$ there exist
radii $r_1>0$ and a threshold $t_0>0$ such that, for all $|t|\le t_0$,
\[
g_t:=\phi+t\,h\quad\text{satisfies}\quad g_t^{**}(x)=g_t(x)
\quad\text{and}\quad \nabla g_t^{**}(x)=\nabla g_t(x).
\]
(In particular, the convexification is “inactive” at $x$, so the first-order
expansion $\nabla g_t^{**}(x)=\nabla\phi(x)+t\,\nabla h(x)+o(t)$ holds.)

Below is a \textbf{fully explicit proof} that uses only classical estimates,
Carathéodory convex combinations (finite supports), and the local strict
convexity of $\phi$ at $x$.

\medskip
\noindent\rule{\textwidth}{0.4pt}

\subsection*{Step 1 — Local strict convexity of $\phi$ at $x$}

By Alexandrov at $x$ and $H\succ0$, for any $\varepsilon\in(0,\lambda_*/2)$,
with $\lambda_*:=\lambda_{\min}(H)>0$, there exists $r_1>0$ such that
\[
\phi(x+u)\ \ge\ \phi(x)+\langle \nabla\phi(x),u\rangle
+\frac{\lambda_*-\varepsilon}{2}\,\|u\|^2,\qquad \|u\|\le r_1.
\tag{1}
\]
In particular, $\phi$ is \textbf{$m$-strongly convex} on $B(x,r_1)$ with
$m:=\lambda_*-\varepsilon>0$.

\authors{Incorrect again}

By continuity of the function
$u\mapsto \phi(x+u)-\phi(x)-\langle\nabla\phi(x),u\rangle$ and the strict
positivity of the quadratic term in (1), the \textbf{boundary gap}
\[
\delta_1\ :=\ \min_{\|u\|=r_1}\,\Big[\phi(x+u)-\phi(x)-\langle\nabla\phi(x),u\rangle\Big]
\]
is strictly positive: $\delta_1>0$. By convexity along rays, this implies
\[
\phi(y)\ \ge\ \phi(x)+\langle\nabla\phi(x),y-x\rangle+\delta_1
\quad\text{for all }\ \|y-x\|\ge r_1.
\tag{2}
\]

\medskip
\noindent\rule{\textwidth}{0.4pt}

\subsection*{Step 2 — Two global quadratic inequalities for $h$}

Since $h\in C_c^2$, its Hessian is bounded globally:
$\|\nabla^2 h\|_\infty\le M$ for some $M<\infty$. Hence the standard
“M-smooth” Taylor bounds hold \textbf{for all} $y\in\mathbb R^d$:
\[
\begin{aligned}
h(y)&\ge h(x)+\langle \nabla h(x),y-x\rangle-\frac{M}{2}\|y-x\|^2,\\
h(y)&\le h(x)+\langle \nabla h(x),y-x\rangle+\frac{M}{2}\|y-x\|^2.
\end{aligned}
\tag{3}
\]

\medskip
\noindent\rule{\textwidth}{0.4pt}

\subsection*{Step 3 — A Carathéodory/“mixture” lower bound for $g_t^{**}(x)$}

By finite-dimensional convex analysis (Carathéodory), one has
\[
g_t^{**}(x)\ =\ \inf\Big\{\textstyle\sum_{i=1}^{m}\lambda_i\,g_t(y_i)\ :\
\sum_{i=1}^m\lambda_i=1,\ \sum_{i=1}^m\lambda_i y_i=x,\ m\le d+1\Big\}.
\tag{4}
\]
Fix any such finite decomposition of $x$ and split the index set into
\[
I_{\rm in}:=\{i:\ \|y_i-x\|<r_1 \},\qquad I_{\rm out}:=\{i:\ \|y_i-x\|\ge r_1 \},
\]
with total outside mass $\Lambda_{\rm out}:=\sum_{i\in I_{\rm out}}\lambda_i$.

We now \textbf{lower-bound} the average $\sum \lambda_i g_t(y_i)$ in terms of
$\Lambda_{\rm out}$ and the “variance” of the inside points.

For $i\in I_{\rm in}$, combine (1) and the left inequality in (3):
\[
\begin{aligned}
g_t(y_i)
&= \phi(y_i)+t\,h(y_i)\\
&\ge \phi(x)+t\,h(x)+\langle \nabla\phi(x)+t\nabla h(x),y_i-x\rangle
+ \frac{m-|t|M}{2}\,\|y_i-x\|^2.
\end{aligned}
\tag{5}
\]

For $i\in I_{\rm out}$, use (2) and the same left inequality in (3):
\[
\begin{aligned}
g_t(y_i)
&= \phi(y_i)+t\,h(y_i)\\
&\ge \phi(x)+\langle \nabla\phi(x),y_i-x\rangle+\delta_1
+ t\big(h(x)+\langle \nabla h(x),y_i-x\rangle\big)
- \frac{|t|M}{2}\,\|y_i-x\|^2.
\end{aligned}
\tag{6}
\]

Summing (5)–(6) with weights $\lambda_i$, the \textbf{barycenter constraint}
$\sum \lambda_i(y_i-x)=0$ cancels all linear terms, giving
\[
\sum_{i}\lambda_i g_t(y_i)
\ \ge\ (\phi(x)+t\,h(x))
+ \frac{m-|t|M}{2}
\sum_{i\in I_{\rm in}}\lambda_i\|y_i-x\|^2
+ \Lambda_{\rm out}\Big(\delta_1 - \frac{|t|M}{2}r_1^2\Big).
\tag{7}
\]

\medskip
\noindent\rule{\textwidth}{0.4pt}

\subsection*{Step 4 — Choose $t_0$ so both extra terms are nonnegative}

Pick
\[
t_0\ :=\ \min\Big\{\frac{m}{M},\ \frac{2\delta_1}{M r_1^2}\Big\}>0.
\]
Then for all $|t|\le t_0$,
\[
\frac{m-|t|M}{2}\ \ge\ 0,\qquad
\delta_1 - \frac{|t|M}{2}r_1^2\ \ge\ 0.
\]
Hence from (7),
\[
\sum_{i}\lambda_i g_t(y_i)\ \ge\ g_t(x)
\quad\text{for every admissible convex decomposition in (4)}.
\]
Taking the infimum over all such decompositions gives
\[
g_t^{**}(x)\ \ge\ g_t(x).
\]
Since always $g_t^{**}\le g_t$, we conclude
\[
\boxed{\,g_t^{**}(x)=g_t(x)\,}\qquad\text{for all }\ |t|\le t_0.
\]

Moreover, equality in (7) forces $\Lambda_{\rm out}=0$ and
$\sum_{i\in I_{\rm in}}\lambda_i\|y_i-x\|^2=0$, i.e., the minimizing
decomposition is the trivial one $y_1=x$. That implies
$\partial g_t^{**}(x)= {\nabla g_t(x) }$ and hence
\[
\boxed{\,\nabla g_t^{**}(x)=\nabla g_t(x)\,}
\]
for $|t|\le t_0$. In particular,
\[
\nabla g_t^{**}(x)\ =\ \nabla\phi(x)+t\,\nabla h(x)\quad\text{to first order.}
\]

\medskip
\noindent\rule{\textwidth}{0.4pt}

\subsection*{What we used (and why it’s legitimate)}

\begin{itemize}
    \item \textbf{Alexandrov + $\nu_0$-a.e.\ strictness:} At $\mu_0$-a.e.\ $x$, $\nabla^2\phi(x)$
    exists and is positive definite (by differentiating
    $\nabla\phi^*(\nabla\phi(x))=x$ at those points where both Hessians exist).
    This gives (1) and the \textbf{strict boundary gap} $\delta_1>0$ in (2).
    \item \textbf{Global $C^2$ control on $h$:} $h\in C_c^2$ implies a \textbf{global}
    Hessian bound $M$, hence the quadratic Taylor bounds (3) hold \textbf{for all}
    points $y$.
    \item \textbf{Carathéodory representation:} In finite dimension, the biconjugate at a
    point equals the infimum of \textbf{finite} convex combinations of function values
    with that barycenter. This allowed us to reduce to a simple algebraic estimate
    (7) that cleanly separates “inside” and “outside” contributions.
    \item \textbf{Smallness of $t$:} Picking $t_0$ as in Step 4 ensures the positive
    terms (coming from local strong convexity of $\phi$ and the boundary gap
    $\delta_1$) dominate the possible negative quadratic defect coming from the
    $t\,h$-perturbation.
\end{itemize}

\medskip
\noindent\rule{\textwidth}{0.4pt}

\subsection*{Takeaway}

The local quadratic bound does \textbf{not} automatically give $g_t=g_t^{**}$
on an entire ball, but it \textbf{does} guarantee equality \textbf{at the point $x$} for
all small $|t|$. 

\authors{Now, the statement is correct, but the proof is not. This is an important progress in our opinion: the most difficult part to prove Equation~\eqref{eq:us} was actually to make the right conjecture (\textit{i.e.}, the fact that $o(t) = 0$ for $t$ small enough, depending on $x$). Once the right conjecture was written, the result was rather easy to prove.}

That is exactly what we need: the biconjugate is inactive at
$x$ to first order, hence
\[
\nabla(\phi+t h)^{**}(x)=\nabla\phi(x)+t\,\nabla h(x)+o(t),
\]
which is the key input for the velocity computation and the entropy variation.

\medskip

\end{ChatGPTBox} \begin{AuthorsBox}   Can we lower bound the hessian a.e on a ball around $x$?

\medskip

\end{AuthorsBox} \begin{ChatGPTBox}       Short answer: \textbf{in general, no.}  
From $ \nabla^2\phi(x)=H\succ0 $ at a \emph{single} Alexandrov point $x$, you \textbf{cannot} conclude there exists $r>0$ and $m>0$ such that
\[
\nabla^2\phi(y)\succeq m I \quad\text{for a.e. }y\in B(x,r).
\]
Convexity gives $ \nabla^2\phi(y)\succeq 0 $ a.e., but a \textbf{strictly positive uniform lower bound a.e. on a whole ball} need not hold without extra smoothness (e.g. $C^2$ in a neighborhood).

\medskip
\noindent\rule{\textwidth}{0.4pt}

\subsection*{Why a uniform a.e. lower bound need not hold}

\begin{itemize}
    \item For a convex function $ \phi$, the distributional Hessian $D^2\phi$ is a symmetric \emph{positive} matrix-valued measure. The a.e. “Alexandrov Hessian” $ \nabla^2\phi \in L^1_{\rm loc}$ is the density of its absolutely continuous part; there may be a singular part (e.g. in 1D, $\phi(x)=|x|$ has a delta at $0$).  
    \item At a point $x$ where $\phi$ is twice Fréchet differentiable and $H=\nabla^2\phi(x)\succ0$, you have the \emph{pointwise} second-order expansion, but this does \textbf{not} force $ \nabla^2\phi(y)$ to stay uniformly away from $0$ a.e. on an entire neighborhood.  
    Indeed, in 1D one can construct $\phi$ with $\phi''(0)=1$ while in every neighborhood of $0$ there is a set of positive measure where $\phi''=0$.  

    \textit{Sketch:} pick a measurable set $E\subset(-r,r)$ with zero Lebesgue density at $0$ but $|E\cap(-\rho,\rho)|>0$ for all $\rho>0$. Let $g=1_{(-r,r)\setminus E}\ge0$ and define $\phi(x)=\int_0^x\!\int_0^s g(t)\,dt\,ds$. Then $\phi$ is convex, $\phi''=g$ a.e., $\phi''(0)$ exists and equals $1$ (Lebesgue point), but in every neighborhood $\phi''$ vanishes on a set of positive measure. Thus no uniform $m>0$ bound a.e. on a whole ball.
\end{itemize}

So without continuity of $\nabla^2\phi$ (e.g. $\phi\in C^2$ near $x$), you cannot expect such a bound.

\medskip
\noindent\rule{\textwidth}{0.4pt}

\subsection*{What \emph{does} hold (and is enough for our proof)}

You can get \textbf{approximate lower bounds “in measure”} by Lebesgue differentiation:

\textbf{Approximate lower bound.}  
If $\phi$ is twice differentiable at $x$ with $H\succ0$ and $\lambda_*=\lambda_{\min}(H)$, then for every $\varepsilon,\eta>0$ there exists $r>0$ such that
\[
\big|\big\{y\in B(x,r):\ \lambda_{\min}\big(\nabla^2\phi(y)\big)\ge \lambda_*-\varepsilon\big\}\big|
\ \ge\ (1-\eta)\,|B(x,r)|.
\]
\textit{Reason:} each matrix entry of $\nabla^2\phi$ is in $L^1_{\rm loc}$ and has value $H$ at $x$ as a Lebesgue differentiation point; for small balls, the average matrix is close to $H$. By Chebyshev and continuity of eigenvalues in the entries, all but an $\eta$-fraction of points have $\nabla^2\phi(y)$ within $\varepsilon$ of $H$, hence $\lambda_{\min}(\nabla^2\phi(y))\ge \lambda_*-\varepsilon$.

This “almost everywhere up to small proportion” is the right statement at the level of general convexity.

\medskip
\noindent\rule{\textwidth}{0.4pt}

\subsection*{Crucially, we never needed a uniform a.e. bound}

In our Route 1 proof:
\begin{itemize}
    \item We used only the \textbf{pointwise} second-order expansion at $x$ and $\lambda_*>0$ to get the \textbf{one-point} quadratic lower bound
    \[
    \phi(x+u)\ \ge\ \phi(x)+\langle \nabla\phi(x),u\rangle+\tfrac{\lambda_*-\varepsilon}{2}\|u\|^2
    \quad\text{for small }u,
    \]
    which is a statement \textbf{around the base point $x$}.
    \item That one-point inequality sufficed to prove the “convexification is inactive at $x$” for $(\phi+th)^{**}$, via a Carathéodory/variance estimate on convex combinations whose \textbf{barycenter is $x$}. No global strong convexity (and hence no uniform Hessian lower bound on the ball) was required.
\end{itemize}

\medskip
\noindent\rule{\textwidth}{0.4pt}

\subsection*{When can you \emph{do} it?}

If you assume \textbf{$C^2$ regularity} (or even continuity of $\nabla^2\phi$) on a neighborhood of $x$, then yes: by continuity there exists $r>0$ such that
\[
\nabla^2\phi(y)\ \succeq\ (\lambda_*-\varepsilon)\,I \quad\text{for all }y\in B(x,r),
\]
which yields genuine \textbf{strong convexity on the ball} and the classical Jensen–variance inequality for all barycenters $z\in B(x,r)$. But this is an extra hypothesis we did not (and do not need to) impose.

\medskip
\noindent\rule{\textwidth}{0.4pt}

\subsection*{Bottom line}

\begin{itemize}
    \item \textbf{No} general uniform a.e. lower bound on a whole ball can be deduced from $ \nabla^2\phi(x)\succ0$ at a single point.  
    \item You \textbf{can} get an approximate lower bound in measure via Lebesgue differentiation.  
    \item Our argument is designed to avoid needing any such uniform bound: the one-point quadratic lower bound at $x$ is enough to keep the convexification inactive at $x$ and to continue with the entropy variation.
\end{itemize}

\medskip

\end{ChatGPTBox} \begin{AuthorsBox}   OK. Is there a reference for the expansion of $(\phi + th)^{**}$? Because we know it is smaller than $(\phi + th)$ so I expect some mathematicians studied the order of the difference between the two? Or maybe a reference on local equality between the function and its biconjugate?

\medskip

\end{AuthorsBox} \begin{ChatGPTBox}       Great question. There isn’t (to my knowledge) a standalone “named theorem” that gives a ready-made \textbf{asymptotic expansion for $(\phi+th)^{**}$} in general, but there are two standard routes that together yield exactly what you need:

\medskip
\noindent\rule{\textwidth}{0.4pt}

\subsection*{1) Local equality with the convex envelope near points of strong local convexity}

\textbf{Fact (local inactivity of convexification).}  
Let $\phi:\mathbb{R}^d\to\mathbb{R}\cup\{+\infty\}$ be proper, lsc, convex. Fix $x_0\in\mathrm{int}(\mathrm{dom}\,\phi)$ with $\phi\in C^2$ at $x_0$ and $\nabla^2\phi(x_0)\succ0$. Then there exist $r>0$ and $t_0>0$ such that for all $|t|\le t_0$,
\[
(\phi+th)^{**}(x)=\phi(x)+t\,h(x)\quad\text{for all }x\in B_r(x_0),
\]
whenever $h$ is $C^2$ on a neighborhood of $x_0$.

\textbf{Why this is true.}  
Positive definiteness at $x_0$ implies a \textbf{local quadratic lower bound}
\[
\phi(x_0+u)\;\ge\;\phi(x_0)+\langle\nabla\phi(x_0),u\rangle+\tfrac{m}{2}\|u\|^2
\quad\text{for small }\|u\|
\]
for some $m>0$. If $\|\nabla^2 h\|_\infty$ is bounded near $x_0$, then for $|t|$ small enough we still have
\[
\phi(x)+t\,h(x)\;\ge\;\text{(affine support at $x_0$)}+\tfrac{m'}{2}\|x-x_0\|^2
\]
on a small ball around $x_0$ with $m'>0$. Hence $\phi+th$ is \textbf{convex on that ball}, \authors{Same error again} so its \textbf{convex envelope coincides with itself} there. 

This is a one-line consequence of the \textbf{Fenchel–Moreau theorem} (biconjugate is the lsc convex envelope) together with the representation of convex envelopes via \textbf{Carathéodory’s theorem} (any point in the convex hull uses at most $d+1$ points): if a function is already convex on a neighborhood $U$, then its convex envelope equals the function on a slightly smaller neighborhood $U'\subset U$. For background, see:
\begin{itemize}
    \item Rockafellar, \textit{Convex Analysis} (1970), Ch. 2: Fenchel–Moreau Theorem;
    \item Rockafellar–Wets, \textit{Variational Analysis} (1998), §2 and §11;
    \item Borwein–Lewis, \textit{Convex Analysis and Nonlinear Optimization} (Springer, 2006), §2.3–2.4.
\end{itemize}

\textbf{Consequence (first-order expansion).}  
From local equality,
\[
(\phi+th)^{**}(x_0)=\phi(x_0)+t\,h(x_0)\quad\text{for all $|t|\le t_0$,}
\]
hence $(\phi+th)^{**}(x_0)=\phi(x_0)+t\,h(x_0)+o(t)$ with actually zero error for $|t|$ small. Differentiating in $x$, you likewise get
\[
\nabla\big((\phi+th)^{**}\big)(x_0)=\nabla\phi(x_0)+t\,\nabla h(x_0)
\quad\text{for $|t|$ small,}
\]
because gradients match at points where the convex envelope equals the original function and both are $C^1$.

This is the simplest route if, as in your application, you only need the behavior at a.e. $x$ where $\phi$ is $C^2$ and $\nabla^2\phi\succ0$ (which is a \textbf{local} assumption, not a global strong convexity).

\authors{[...]}

\medskip

\end{ChatGPTBox} \begin{AuthorsBox}   I don't understand why $\phi+th$ is convex on that ball in 1. We should have the inequality for every $x$ in the ball.  
I don't understand how you obtain expansion of the biconjugate via Fenchel–Moreau in 2.

Please check. If you made no error, explain in details

\medskip

\end{AuthorsBox} \begin{ChatGPTBox}       You’re right to press on both points. Let me (i) correct the earlier wording about “convex on the ball”, and (ii) give a \textbf{fully detailed, self-contained proof} that for small $|t|$ we have
\[
(\phi+th)^{**}(x)=\phi(x)+t\,h(x)
\quad\text{and}\quad
\nabla(\phi+th)^{**}(x)=\nabla\phi(x)+t\,\nabla h(x),
\]
\textbf{at the single base point $x$} where $\phi$ has a positive-definite Alexandrov Hessian. This is all we need for the first variation.

\medskip
\noindent\rule{\textwidth}{0.4pt}

\subsection*{1) Clarification: we do \textbf{not} need $ \phi+th $ convex on a whole ball}

What we used—and what is true from Alexandrov— is a \textbf{one-point quadratic lower support} at $x$:
\[
\phi(x+u)\ \ge\ \phi(x)+\langle\nabla\phi(x),u\rangle+\frac{\lambda_*-\varepsilon}{2}\,\|u\|^2
\quad\text{for small }u,
\]
where $\lambda_*=\lambda_{\min}(\nabla^2\phi(x))>0$. This does \textbf{not} say $\phi$ (or $\phi+th$) is strongly convex on an entire ball; 
\authors{The model finally corrected its mistake, after suggestion}
and we don’t need that. The one-point inequality is enough to rule out any improvement by convexification \textbf{at the single point $x$}, via a Carathéodory/variance argument.

So please ignore the earlier shorthand “$\phi+th$ is strongly convex on $B$”; the precise and sufficient statement is the \textbf{one-point} bound above.

\medskip
\noindent\rule{\textwidth}{0.4pt}

\subsection*{2) Local equality with the biconjugate at $x$ (and gradient equality)}

\textbf{Setting.} Let $\phi:\mathbb{R}^d\to(-\infty,+\infty]$ be proper, lsc, convex. Fix a point $x\in\mathrm{int}(\mathrm{dom}\,\phi)$ where $\phi$ is twice Fréchet differentiable with $H:=\nabla^2\phi(x)\succ0$. Let $h\in C^2(\mathbb{R}^d)$ (compact support is fine but not essential here), and define $g_t:=\phi+th$.

\textbf{Goal.} Show that for all sufficiently small $|t|$,
\[
g_t^{**}(x)=g_t(x)
\quad\text{and}\quad
\nabla g_t^{**}(x)=\nabla g_t(x)=\nabla\phi(x)+t\,\nabla h(x).
\]

\subsubsection*{Step 2.1 — One-point quadratic lower support and “boundary gap”}

From Alexandrov at $x$, for any $\varepsilon\in(0,\lambda_*/2)$ there is $r_1>0$ such that for all $\|u\|\le r_1$,
\[
\phi(x+u)\ \ge\ \phi(x)+\langle\nabla\phi(x),u\rangle+\frac{m}{2}\|u\|^2,
\quad m:=\lambda_*-\varepsilon>0. \tag{LQ}
\]
By continuity of $u\mapsto \phi(x+u)-\phi(x)-\langle\nabla\phi(x),u\rangle$, the \textbf{boundary gap}
\[
\delta_1\ :=\ \min_{\|u\|=r_1}\big [\phi(x+u)-\phi(x)-\langle\nabla\phi(x),u\rangle\big ]
\]
is strictly positive. By convexity along rays, this yields
\[
\phi(y)\ \ge\ \phi(x)+\langle\nabla\phi(x),y-x\rangle+\delta_1
\quad\text{for all }\ \|y-x\|\ge r_1. \tag{BG}
\]

For $h$, let $M:=\sup_{B(x,r_1)}\|\nabla^2 h\|$. Then the standard quadratic Taylor bounds give, for all $y$ with $\|y-x\|\le r_1$,
\[
h(y)\ \ge\ h(x)+\langle\nabla h(x),y-x\rangle-\frac{M}{2}\|y-x\|^2, \tag{T$_-$}
\]
and globally,
\[
h(y)\ \ge\ h(x)+\langle\nabla h(x),y-x\rangle-\frac{M}{2}\|y-x\|^2
\quad\text{(always valid with this $M$ on the ball).} \tag{T$_-$' }
\]

\subsubsection*{Step 2.2 — Carathéodory/variance lower bound for $g_t^{**}(x)$}

In $\mathbb{R}^d$, Carathéodory implies that the convex envelope at a point can be computed using \textbf{finite} convex combinations: for any $\eta>0$ there are $m\le d+1$, points $y_i$, and coefficients $\lambda_i\ge 0$ summing to 1 with $\sum\lambda_i y_i=x$ such that
\[
g_t^{**}(x)\ \ge\ \sum_{i=1}^m \lambda_i\,g_t(y_i) - \eta.
\]
(You can take $\eta=0$ if you accept “infimum over all such finite representations”.)

Split indices into
\[
I_{\rm in}:=\{i:\ \|y_i-x\|<r_1 \},\qquad I_{\rm out}:=\{i:\ \|y_i-x\|\ge r_1 \},
\]
and let $\Lambda_{\rm out}:=\sum_{i\in I_{\rm out}}\lambda_i$.

\begin{itemize}
    \item For $i\in I_{\rm in}$, combine (LQ) and (T$_-$):
    \[
    \begin{aligned}
    g_t(y_i) &= \phi(y_i)+t\,h(y_i) \\
             &\ge \phi(x)+t\,h(x)+\langle\nabla\phi(x)+t\nabla h(x),y_i-x\rangle
                 +\frac{m-|t|M}{2}\,\|y_i-x\|^2.
    \end{aligned}
    \]
    \item For $i\in I_{\rm out}$, use (BG) and (T$_-$'):
    \[
    \begin{aligned}
    g_t(y_i)&\ge \phi(x)+\langle\nabla\phi(x),y_i-x\rangle+\delta_1
                 + t\big(h(x)+\langle\nabla h(x),y_i-x\rangle\big)
                 -\frac{|t|M}{2}\,\|y_i-x\|^2. 
    \end{aligned}
    \]
\end{itemize}

Sum with weights $\lambda_i$; the barycenter constraint $\sum\lambda_i(y_i-x)=0$ cancels all linear terms. Using $\|y_i-x\|\ge r_1$ for $i\in I_{\rm out}$, we obtain
\[
\sum_i \lambda_i g_t(y_i)
\ \ge\ g_t(x)
  + \frac{m-|t|M}{2}\,\sum_{i\in I_{\rm in}}\lambda_i\|y_i-x\|^2
  + \Lambda_{\rm out}\Big(\delta_1-\tfrac{|t|M}{2}\,r_1^2\Big). \tag{*}
\]
\authors{But, one cannot lower bound $-\tfrac{|t|M}{2}\,\|y_i - x\|^2 \geq -\tfrac{|t|M}{2}\,r_1^2$ because $y_i$ is outside the ball ($\|y_i-x\|\ge r_1$). Therefore, the end of the proof is wrong.}

Choose
\[
t_0\ :=\ \min\Big \{\frac{m}{M},\,\frac{2\delta_1}{M r_1^2}\Big \}>0.
\]
Then for all $|t|\le t_0$, both extra terms on the right of (*) are $\ge 0$, hence $\sum_i \lambda_i g_t(y_i)\ge g_t(x)$. Taking the infimum over all such finite convex representations gives
\[
g_t^{**}(x)\ \ge\ g_t(x).
\]
Since always $g_t^{**}\le g_t$, we conclude
\[
\boxed{\,g_t^{**}(x)=g_t(x)\quad\text{for all }|t|\le t_0. \,}
\]

Moreover, equality in (*) forces $\Lambda_{\rm out}=0$ and
$\sum_{i\in I_{\rm in}}\lambda_i\|y_i-x\|^2=0$, i.e. the only minimizing representation is the trivial one $y_1=x$. This implies $\partial g_t^{**}(x)= {\nabla\phi(x)+t\nabla h(x) }$ (see next step), i.e. gradient equality.

\subsubsection*{Step 2.3 — Gradient equality at $x$}

Since $g_t^{**}\le g_t$ and $g_t^{**}(x)=g_t(x)$, we have the subdifferential inclusion
\[
\partial g_t^{**}(x)\ \subset\ \partial g_t(x).
\]
(Reason: if $p\in\partial g_t^{**}(x)$, then $g_t^{**}(z)\ge g_t^{**}(x)+\langle p,z-x\rangle$, so $g_t(z)\ge g_t^{**}(z)\ge g_t(x)+\langle p,z-x\rangle$, i.e. $p\in\partial g_t(x)$.)

But $g_t$ is differentiable at $x$ (sum of differentiable $\phi$ at $x$ and $h$), so $\partial g_t(x)= {\nabla g_t(x) }= {\nabla\phi(x)+t\nabla h(x) }$. Thus $\partial g_t^{**}(x)$ is a singleton and equals that same vector; in particular $g_t^{**}$ is differentiable at $x$ and
\[
\boxed{\,\nabla g_t^{**}(x)=\nabla g_t(x)=\nabla\phi(x)+t\,\nabla h(x).\,}
\]
\authors{[...]}

\medskip

\end{ChatGPTBox} \begin{AuthorsBox}   The problem is that I don't trust the biconjugate inactive lemma because it would mean that $\nabla(\phi + t h)$ is monotone and therefore $\phi + t h$ is convex.

\authors{Here we didn't realize that the model's statement was pointwise at $x$, and not on a neighborhood of $x$. But the model's proof is still wrong at this stage (because of the Carathéodory argument, see our comment above).}

\medskip

\end{AuthorsBox} \begin{ChatGPTBox}       You’re absolutely right to be cautious: \textbf{“biconjugate inactive at $x$” does *not* mean $\nabla(\phi+t h)$ is (cyclically) monotone on a neighborhood, nor that $\phi+t h$ is convex there.}  

What we proved—and what we actually need—is strictly weaker:

\medskip
\textbf{Pointwise inactivity at the base point.}  
If $\phi$ is twice Fréchet differentiable at $x$ with $\nabla^2\phi(x)\succ0$ and $h\in C^2$, then for $|t|$ small
\[
(\phi+t h)^{**}(x)=\phi(x)+t\,h(x)\quad\text{and}\quad
\nabla(\phi+t h)^{**}(x)=\nabla\phi(x)+t\,\nabla h(x).
\]
This tells you what happens \textbf{at the point $x$} (and for the derivative at $t=0$), not on a whole neighborhood, and certainly not globally.

\medskip
\noindent\rule{\textwidth}{0.4pt}

\subsection*{Why this does \textbf{not} force $\phi+t h$ to be convex (or $\nabla(\phi+t h)$ to be monotone)}

\begin{itemize}
    \item Convexity (equivalently, monotonicity of the gradient when $C^1$) is a \textbf{pairwise/global} property:
    \[
    \langle \nabla(\phi+t h)(x)-\nabla(\phi+t h)(z),\,x-z\rangle\ \ge 0
    \quad\forall x,z.
    \]
    Our statement gives equality of $(\phi+t h)^{**}$ and $\phi+t h$ \textbf{only at the single point $x$} (and equality of their gradients there). It says nothing about values at other points $z$, and therefore nothing about all such pairwise inequalities.
    \item It is easy to cook up $C^2$ functions $h$ for which $\phi+t h$ is convex \textbf{near $x$} (thanks to the positive curvature of $\phi$ at $x$) but becomes \textbf{non-convex far away}. In that case $(\phi+t h)^{**}$ coincides with $(\phi+t h)$ at $x$ (and even on a tiny neighborhood if you assume a bit more regularity), yet $\phi+t h$ is not convex globally and its gradient is not a monotone map on $\mathbb{R}^d$.
\end{itemize}

\medskip
\noindent\rule{\textwidth}{0.4pt}

\subsection*{What the lemma actually proved (and how it avoids any hidden convexity assumption)}

\begin{enumerate}
    \item \textbf{One-point quadratic lower support at $x$:}  
    From Alexandrov and $\nabla^2\phi(x)\succ0$, for small $\|u\|$,
    \[
    \phi(x+u)\ \ge\ \phi(x)+\langle\nabla\phi(x),u\rangle+\tfrac{m}{2}\|u\|^2,
    \quad m>0.
    \]
    This is a statement about $x$, not about every point in a ball.

    \item \textbf{Carathéodory/variance estimate at the barycenter $x$:}  
    For any \textbf{finite} convex combination $\sum_i \lambda_i y_i=x$,
    the above one-point inequality and a quadratic bound on $h$ imply
    \[
    \sum_i \lambda_i\big(\phi(y_i)+t h(y_i)\big)\ \ge\ \phi(x)+t h(x),
    \]
    provided $|t|$ is small enough. (Outside points cost a fixed “boundary gap”; inside points cost a quadratic “variance”.)

    \item \textbf{Take the infimum over such finite mixtures:}  
    By Fenchel–Moreau (biconjugate = convex envelope) and Carathéodory (finite support suffices in $\mathbb{R}^d$), we get
    \[
    (\phi+t h)^{**}(x)\ \ge\ \phi(x)+t h(x).
    \]
    Since always $(\phi+t h)^{**}\le \phi+t h$, this yields \textbf{equality at $x$}.

    \item \textbf{Subgradient inclusion $\Rightarrow$ gradient equality at $x$:}  
    From $f^{**}\le f$ and $f^{**}(x)=f(x)$ one has $\partial f^{**}(x)\subset\partial f(x)$. Here $f=\phi+t h$ is differentiable at $x$, so $\partial f(x)= {\nabla f(x) }$, hence $\partial f^{**}(x)= {\nabla f(x) }$ and
    \[
    \nabla(\phi+t h)^{**}(x)=\nabla(\phi+t h)(x)=\nabla\phi(x)+t\nabla h(x).
    \]
\end{enumerate}

\medskip
\noindent\rule{\textwidth}{0.4pt}

\subsection*{Notice how no convexity on a neighborhood is assumed}

Nowhere did we assume $\phi+t h$ is convex on a ball, nor that $\nabla(\phi+t h)$ is monotone. We only used:

\begin{itemize}
    \item a \textbf{one-point} curvature lower bound at $x$,
    \item a \textbf{finite-mixture} (Carathéodory) representation to control convexification \textbf{at that barycenter}, and
    \item the general inequality chain $f^{**}\le f$ paired with subdifferential inclusion at a point of equality.
\end{itemize}

That’s exactly tailored to what we need for the first variation:
we vary along $\phi_t=(\phi+t h)^{**}$, and at $\mu_0$-a.e.\ point $x$ the \textbf{directional derivative in $t$} of $\nabla\phi_t(x)$ is $\nabla h(x)$. Passing to Eulerian variables gives the velocity $v_0(y)=\nabla h(\nabla\phi^*(y))$ $\nu_0$-a.e., and the rest of the argument (OU smoothing after pushforward $\rightarrow$ classical chain rule $\rightarrow$ Stein identity $\rightarrow$ $\nu_0=\pi$) goes through—without ever needing global convexity of $\phi+t h$ or monotonicity of its gradient.

If you \textit{do} want \textbf{equality on a whole neighborhood} (not just at $x$), you need extra regularity: e.g., if $\nabla^2\phi$ is continuous at $x$ so that $\nabla^2\phi(y)\succeq (\lambda_{\min}(\nabla^2\phi(x))-\varepsilon)I$ for all $y$ in a small ball, then for small $|t|$ you can ensure $\phi+t h$ is \textbf{strongly convex} on that ball and hence coincides with its biconjugate there. But that strengthening is \textbf{not required} for the proof—pointwise equality at $x$ is sufficient.

\medskip

\end{ChatGPTBox} \begin{AuthorsBox}   I see. You proved that $\nabla (\phi + t h)^{**}(x) = \nabla (\phi + t h)(x)$ for a.e.\ $x$, but how $t$ is small depends on $x$.

So really $t = t(x)$.

\medskip

\end{AuthorsBox} \begin{ChatGPTBox}       You’re absolutely right: with the \textbf{local convex-envelope lemma} we proved,
for each \emph{fixed} $x$ in a full $\mu_0$-measure set there exists a (pointwise)
radius $t_x>0$ such that
\[
\nabla\big(\phi_t\big)(x)=\nabla(\phi+t h)^{**}(x)
=\nabla\phi(x)+t\,\nabla h(x)\qquad\text{for all }|t|\le t_x,
\]
but $t_x$ \textbf{does depend on $x$}. Below I explain why this dependence causes
no problem for the variation we need, and how to make the key steps rigorous
without any uniform $t$-window.

\authors{[...]}

\end{ChatGPTBox} \begin{AuthorsBox}   Do you have a simpler proof for the local convex envelope lemma?

\medskip

\end{AuthorsBox} \begin{ChatGPTBox}       Yes—there’s a short “one-page” proof that avoids the epigraph geometry and keeps only the bare minimum of estimates. It proves \textbf{pointwise} inactivity of the convex envelope at your base point $x$ (so it never claims $\phi+th$ is convex on a neighborhood).

\medskip
\noindent\rule{\textwidth}{0.4pt}

\subsection*{Local convex–envelope lemma (short proof)}

\textbf{Claim.}  
Let $\phi:\mathbb R^d\to(-\infty,+\infty]$ be proper, lsc, convex. Fix a point $x\in\mathrm{int}(\mathrm{dom}\,\phi)$ where $\phi$ is twice Fréchet differentiable with $H:=\nabla^2\phi(x)\succ0$. Let $h\in C^2(\mathbb R^d)$ (compact support is fine). Then there exists $t_x>0$ such that, for all $|t|\le t_x$,
\[
(\phi+t h)^{**}(x)=\phi(x)+t\,h(x)
\quad\text{and}\quad
\nabla(\phi+t h)^{**}(x)=\nabla\phi(x)+t\,\nabla h(x).
\]

\textbf{Why this does *not* assert convexity of $\phi+th$:} It only claims equality \textbf{at the single point $x$}; it says nothing about other points, so it doesn’t force monotonicity of $\nabla(\phi+th)$.

\medskip
\noindent\rule{\textwidth}{0.4pt}

\subsubsection*{Step 1 — One–point quadratic support for $\phi$; boundary gap}

Let $\lambda_*:=\lambda_{\min}(H)>0$. From Alexandrov at $x$, for any small $\varepsilon\in(0,\lambda_*/2)$ there exists $r>0$ such that
\[
\phi(x+u)\ \ge\ \phi(x)+\langle \nabla\phi(x),u\rangle+\frac{m}{2}\|u\|^2
\quad\text{for }\|u\|\le r, \qquad m:=\lambda_*-\varepsilon>0.
\tag{LQ}
\]
By continuity of $u\mapsto \phi(x+u)-\phi(x)-\langle \nabla\phi(x),u\rangle$ and $m>0$, the “boundary gap”
\[
\delta\ :=\ \min_{\|u\|=r}\Big(\phi(x+u)-\phi(x)-\langle \nabla\phi(x),u\rangle\Big)
\]
satisfies $\delta>0$. Convexity along rays then yields
\[
\phi(y)\ \ge\ \phi(x)+\langle \nabla\phi(x),y-x\rangle+\delta
\quad\text{whenever }\|y-x\|\ge r.
\tag{BG}
\]

Let $M:=\sup_{\|y-x\|\le r}\|\nabla^2 h(y)\|<\infty$. Then the standard quadratic Taylor bound holds on $B(x,r)$:
\[
h(y)\ \ge\ h(x)+\langle \nabla h(x),y-x\rangle-\frac{M}{2}\|y-x\|^2.
\tag{T}
\]

\medskip
\noindent\rule{\textwidth}{0.4pt}

\subsubsection*{Step 2 — Test any convex decomposition of $x$}

By Carathéodory (or the usual finite “barycentric” representation of the convex envelope), it suffices to check all \textbf{finite} convex combinations
\[
x=\sum_{i=1}^m\lambda_i y_i,\qquad \lambda_i\ge 0,\ \sum\lambda_i=1.
\]
Split indices into inside/outside the ball:
\[
I_{\rm in}:=\{i:\ \|y_i-x\|<r \},\qquad I_{\rm out}:=\{i:\ \|y_i-x\|\ge r \},
\]
and denote $\Lambda_{\rm out}:=\sum_{i\in I_{\rm out}}\lambda_i$.

We will show
\[
\sum_{i=1}^m\lambda_i\big(\phi(y_i)+t\,h(y_i)\big)\ \ge\ \phi(x)+t\,h(x),
\tag{*}
\]
for all $|t|$ small enough; taking infimum over such finite decompositions will give $(\phi+th)^{**}(x)\ge \phi(x)+t h(x)$, and the reverse inequality is always true, hence equality.

\begin{itemize}
    \item \textbf{Inside points ($i\in I_{\rm in}$).} From (LQ) and (T),
    \[
    \phi(y_i)+t h(y_i)
    \ \ge\ \phi(x)+t h(x) + \langle \nabla\phi(x)+t\nabla h(x),\,y_i-x\rangle
    + \frac{m-|t|M}{2}\|y_i-x\|^2.
    \]
    \item \textbf{Outside points ($i\in I_{\rm out}$).} From (BG) and (T) (which still holds for those $y_i$ with $\|y_i-x\|\le r$; if not, just use the linear bound $h(y)\ge h(x)+\langle \nabla h(x),y-x\rangle - \tfrac{M}{2}\|y-x\|^2$, which is global if you pick any $M\ge \|\nabla^2 h\|_\infty$ on $\mathbb R^d$ or simply note that $h$ is bounded outside a large ball), we have
    \[
    \phi(y_i)+t h(y_i)
    \ \ge\ \phi(x)+t h(x) + \langle \nabla\phi(x)+t\nabla h(x),\,y_i-x\rangle
    + \delta - \frac{|t|M}{2}\|y_i-x\|^2.
    \]
\end{itemize}

Sum with weights $\lambda_i$. The \textbf{barycenter constraint} $\sum \lambda_i(y_i-x)=0$ cancels all linear terms, and we get
\[
\sum_{i=1}^m\lambda_i\big(\phi(y_i)+t h(y_i)\big)
\ \ge\ \phi(x)+t h(x)
 +\frac{m-|t|M}{2}\sum_{i\in I_{\rm in}}\lambda_i\|y_i-x\|^2
 +\Lambda_{\rm out}\Big(\delta - \frac{|t|M}{2} r^2\Big).
\]
\authors{Same error as above; one cannot lower bound $-\tfrac{|t|M}{2}\,\|y_i - x\|^2 \geq -\tfrac{|t|M}{2}\,r^2$ because $y_i$ is outside the ball ($\|y_i-x\|\ge r$). Therefore, the end of the proof is wrong.}

Choose
\[
t_x\ :=\ \min\Big \{\frac{m}{M},\ \frac{2\delta}{M r^2}\Big \}>0.
\]
Then for all $|t|\le t_x$ both extra terms are $\ge 0$, hence (*) holds. Taking the infimum over all such finite decompositions yields
\[
(\phi+th)^{**}(x)\ \ge\ \phi(x)+t\,h(x).
\]
Since always $(\phi+th)^{**}\le \phi+th$, we have equality at $x$:
\[
(\phi+th)^{**}(x)=\phi(x)+t\,h(x),\qquad |t|\le t_x.
\]

\medskip
\noindent\rule{\textwidth}{0.4pt}

\subsubsection*{Step 3 — Gradient equality at $x$}

From $f^{**}\le f$ and $f^{**}(x)=f(x)$ we always have $\partial f^{**}(x)\subset \partial f(x)$. Here $f=\phi+th$ is differentiable at $x$, so $\partial f(x)= {\nabla\phi(x)+t\nabla h(x) }$. Hence $\partial(\phi+th)^{**}(x)$ is a singleton equal to that vector, i.e.
\[
\nabla(\phi+th)^{**}(x)=\nabla\phi(x)+t\,\nabla h(x).
\]

\medskip
\noindent\rule{\textwidth}{0.4pt}

\subsection*{Why this proof is “simpler”}

\begin{itemize}
    \item It uses only: a one-point quadratic lower support (from Alexandrov), a boundary gap (a compactness minimum on the sphere), the Taylor quadratic bound for $h$, and the \textbf{barycenter identity} $\sum\lambda_i(y_i-x)=0$.
    \item No epigraph geometry, no monotonicity, and no global convexity of $\phi+th$ are assumed or concluded.
    \item The smallness threshold $t_x$ is explicit:
    \[
    t_x\ \le\ \min\!\Big \{\frac{\lambda_{\min}(\nabla^2\phi(x))-\varepsilon}{\sup_{B(x,r)}\|\nabla^2 h\|},\ \frac{2\delta}{r^2\,\sup_{B(x,r)}\|\nabla^2 h\|}\Big \}.
    \]
    (It depends on $x$, which is perfectly fine for the first variation.)
\end{itemize}

If you’d like, I can translate this into a compact lemma + proof you can paste into your draft.

\medskip

\authors{[...]}

\end{ChatGPTBox} \begin{AuthorsBox}   $f$ differentiable at $x$ so its subdiff equal the gradient. This is true for convex $f$ only, no?

\medskip

\end{AuthorsBox} \begin{ChatGPTBox}       Great catch — the statement “$f$ is differentiable at $x$ so $\partial f(x)= {\nabla f(x) }$” is \textbf{only a theorem for convex $f$} (the Rockafellar–Moreau subdifferential is defined for convex functions). For a \textbf{non-convex} $f$, that convex subdifferential isn’t even defined in general.

\textbf{What we actually used does not require $f$ to be convex.} Here’s the precise (convex-analysis–free) argument we need in our setting $f=\phi+t h$ and $g=f^{**}$:

\medskip
\noindent\rule{\textwidth}{0.4pt}

\subsection*{Proposition (what we really use)}
Let $f:\mathbb R^d\to\mathbb R$ be \textbf{Fréchet differentiable at $x$}. Let $g:\mathbb R^d\to(-\infty,+\infty]$ be \textbf{proper, lsc, convex} such that
\[
g\le f\quad\text{everywhere, and}\quad g(x)=f(x).
\]
Assume $\partial g(x)\neq\varnothing$ (true here because $x$ is in the interior of $\mathrm{dom}\,g$). Then
\[
\partial g(x)= {\nabla f(x) }.
\]

\textbf{Proof.} Take any $p\in\partial g(x)$. By the convex subgradient inequality,
\[
g(z)\ \ge\ g(x)+\langle p,z-x\rangle\quad\forall z.
\]
Since $g\le f$ and $g(x)=f(x)$, we get the \textbf{global supporting inequality for $f$}:
\[
f(z)\ \ge\ f(x)+\langle p,z-x\rangle\quad\forall z.\tag{T}
\]
Fix any direction $u$ and set $z=x+\varepsilon u$, $\varepsilon>0$. Divide (T) by $\varepsilon$ and let $\varepsilon\downarrow0$. Differentiability of $f$ at $x$ gives
\[
\langle \nabla f(x),u\rangle\ \ge\ \langle p,u\rangle.
\]
Apply the same with $-u$ to get $\langle \nabla f(x),u\rangle\le \langle p,u\rangle$. Hence
\[
\langle \nabla f(x),u\rangle=\langle p,u\rangle\quad\forall u,
\]
so $p=\nabla f(x)$. Because $\partial g(x)\neq\varnothing$, we conclude
$\partial g(x)= {\nabla f(x) }$. \qed

\medskip
\noindent\rule{\textwidth}{0.4pt}

\subsection*{Why this resolves the concern}

\begin{itemize}
    \item We \textbf{never} used “$\partial f(x)$” for non-convex $f$.  
    We used the fact that any $p\in\partial g(x)$ (with $g$ convex, $g\le f$, equality at $x$) is a \textbf{global supporting slope} for $f$ at $x$, and then the differentiability of $f$ at $x$ forces $p=\nabla f(x)$. This argument does \textbf{not} require $f$ to be convex.
    \item In our application, $f=\phi+th$ is Fréchet differentiable at $x$ (since both $\phi$ and $h$ are), $g=f^{**}$ is convex with $g\le f$, and we proved \textbf{pointwise equality} $g(x)=f(x)$ for $|t|$ small (the local convex-envelope lemma). Also, because $f$ is finite near $x$, $x\in\mathrm{int}(\mathrm{dom}\,g)$, so $\partial g(x)\neq\varnothing$. Thus the proposition gives
    \[
    \nabla(\phi+th)^{**}(x)=\nabla(\phi+th)(x)=\nabla\phi(x)+t\,\nabla h(x),
    \]
    exactly what we needed.
\end{itemize}

\medskip
\noindent\rule{\textwidth}{0.4pt}

\subsection*{Optional note (generalized subdifferentials)}

If you \emph{do} want a notion of subdifferential for non-convex $f$, you can use the \textbf{Fréchet} or \textbf{Clarke} subdifferentials. At points where $f$ is differentiable, \textbf{all} standard generalized subdifferentials collapse to $ {\nabla f(x) }$. We didn’t need this machinery here, because the simple argument above suffices.

\medskip

So: your instinct was right about the convex subdifferential, but the step in our proof is valid because we avoid “$\partial f$” entirely and pin down $\partial g(x)$ via a global supporting inequality plus differentiability of $f$ at $x$.

\medskip

\authors{[...]}

\end{ChatGPTBox} \begin{AuthorsBox}   Ok. Now let's go back to the local equality between the function and its biconjugate.
 
Can we use that the biconjugate is the affine envelope? The only thing to show is that for $t$ small enough there exists an affine minorant that is equal to $f$ at $x$.

\authors{Here, we suggested a proof technique to the model.}

We know that is the case when $t$ is zero, it should still be true when $t$ is small (depending on $x$).

\medskip

\end{AuthorsBox} \begin{ChatGPTBox}       Yes—that “affine–envelope” viewpoint is exactly right and gives a short proof.

Recall the characterization (Fenchel–Moreau, e.g.\ Rockafellar \textit{Convex Analysis} Thm.~12.2): for any extended real function $f$,
\[
f^{**}(x)=\sup\{\,a(x)\ :\ a\text{ affine and }a\le f\text{ on }\mathbb{R}^d\,\}.
\]
So to get $f^{**}(x)=f(x)$ it suffices to \textbf{exhibit one affine minorant $a_t$ of $f$ that touches at $x$}.

We take $f=\phi+t\,h$ with $h\in C^2$ and (in our setup) compactly supported so $\|\nabla^2 h\|_\infty<\infty$. Fix a point $x$ where $\phi$ is twice Fréchet differentiable and $H:=\nabla^2\phi(x)\succ0$. Write
\[
p_0:=\nabla\phi(x),\qquad q:=\nabla h(x),\qquad M:=\|\nabla^2 h\|_\infty.
\]
Define the affine function
\[
a_t(y)\ :=\ f(x)+\langle p_t,\,y-x\rangle,\qquad p_t:=p_0+t\,q.
\]
We will show that for $|t|$ small enough (depending on $x$) we have $a_t(y)\le f(y)$ \textbf{for all $y$}. Since $a_t(x)=f(x)$, this gives $f^{**}(x)\ge a_t(x)=f(x)$, hence $f^{**}(x)=f(x)$.

\medskip
\noindent\rule{\textwidth}{0.4pt}

\subsubsection*{Why $a_t\le f$ for $|t|$ small}

There are two regions to check.

\textbf{1) Near $x$:} (use a local quadratic lower support for $\phi$).  
By Alexandrov at $x$ and $H\succ0$, there exist $r>0$ and $m>0$ such that
\[
\phi(y)\ \ge\ \phi(x)+\langle p_0,\,y-x\rangle+\frac{m}{2}\,\|y-x\|^2
\qquad(\|y-x\|\le r).
\]
Also, by Taylor’s inequality and the global bound $M$,
\[
h(y)\ \ge\ h(x)+\langle q,\,y-x\rangle-\frac{M}{2}\,\|y-x\|^2\qquad(\forall y).
\]
Therefore, for $\|y-x\|\le r$,
\[
\begin{aligned}
f(y)&=\phi(y)+t\,h(y)\\
&\ge f(x)+\langle p_0+t q,\,y-x\rangle+\frac{m-|t|M}{2}\,\|y-x\|^2\\
&\ge a_t(y)\qquad\text{whenever }|t|\le \frac{m}{M}.
\end{aligned}
\]

\textbf{2) Away from $x$:} (use a fixed “boundary gap” of $\phi$).  
By continuity of $u\mapsto \phi(x+u)-\phi(x)-\langle p_0,u\rangle$ and the strict positivity above, there is $\delta>0$ such that
\[
\phi(y)\ \ge\ \phi(x)+\langle p_0,\,y-x\rangle+\delta\qquad(\|y-x\|=r).
\]
Convexity along rays then propagates this outside the ball $B(x,r)$:
\[
\phi(y)\ \ge\ \phi(x)+\langle p_0,\,y-x\rangle+\delta\qquad(\|y-x\|\ge r).
\]
Using again the global Taylor lower bound for $h$,
\[
\begin{aligned}
f(y)&=\phi(y)+t\,h(y)\\
&\ge f(x)+\langle p_0+t q,\,y-x\rangle\;+\;\Big(\delta-\frac{|t|M}{2}\,r^2\Big)
\qquad(\|y-x\|\ge r).
\end{aligned}
\]
\authors{Same error as above; $\|y-x\|\ge r \ge 0$ does not imply $-\|y-x\|^2\ge -r^2$.}

Thus $f(y)\ge a_t(y)$ on $\|y-x\|\ge r$ provided $|t|\le \dfrac{2\delta}{M r^2}$.

Putting 1) and 2) together, for
\[
|t|\ \le\ t_x\ :=\ \min\!\left\{\frac{m}{M},\ \frac{2\delta}{M r^2}\right\}
\]
we have $a_t(y)\le f(y)$ for \textbf{all} $y$, with equality at $x$. Hence $a_t$ is a global affine minorant touching $f$ at $x$, and by the affine-envelope characterization
\[
(\phi+t h)^{**}(x)\ =\ f^{**}(x)\ \ge\ a_t(x)\ =\ f(x).
\]
Since $f^{**}\le f$ always, we conclude $f^{**}(x)=f(x)$.

\medskip
\noindent\rule{\textwidth}{0.4pt}

\subsubsection*{Two quick corollaries (that you also want)}

\begin{enumerate}
    \item \textbf{Gradient equality at $x$.}  
    Let $g:=f^{**}$. From $g\le f$ and $g(x)=f(x)$, any $p\in\partial g(x)$ is a global supporting slope for $f$ at $x$, i.e. $f(z)\ge f(x)+\langle p,z-x\rangle$ for all $z$. Since $f$ is differentiable at $x$, this forces $p=\nabla f(x)$ (take directional difference quotients in $\pm u$). As $\partial g(x)\neq\varnothing$ (interior point), we get
    \[
    \nabla(\phi+t h)^{**}(x)=\nabla(\phi+t h)(x)=\nabla\phi(x)+t\,\nabla h(x).
    \]
    \item \textbf{Only a pointwise $t_x$ is needed.}  
    The bound $t_x$ above depends on $x$ (via $m,\delta,r$), but for our application we only need the \textbf{pointwise} derivative in $t$ at each $x$ (to identify the Lagrangian velocity), not a uniform $t$-window. All subsequent differentiation is done after OU smoothing, where dominated convergence applies.
\end{enumerate}

\medskip
\noindent\rule{\textwidth}{0.4pt}

\subsubsection*{Takeaway}

Yes—you can (and we just did) prove the local equality $f^{**}(x)=f(x)$ for small $|t|$ by using the “biconjugate = affine envelope” characterization. The affine minorant is explicit:
\[
a_t(y)=f(x)+\langle \nabla\phi(x)+t\nabla h(x),\,y-x\rangle,
\]
and the two simple inequalities above show $a_t\le f$ globally when $|t|$ is small (depending on $x$). This gives the desired contact at $x$ and, by the sandwich argument, equality of gradients there as well.

\medskip

\end{ChatGPTBox} \begin{AuthorsBox}   

I don't understand the away from $x$ case. How do you bound by $r$?
\medskip

   \end{AuthorsBox} \begin{ChatGPTBox} 

Good — let’s slow down and focus just on the \textbf{“away from $x$” part} of the local affine–envelope proof.  
The goal there was: for points $y$ not too close to $x$, show that $\phi(y)$ lies strictly above the affine plane
\[
a_0(y) = \phi(x)+\langle \nabla \phi(x), y-x\rangle
\]
by some positive margin. Then, when you perturb by $t h(y)$, you want to check that margin is not eaten up by the small $t$-perturbation.

\medskip
\noindent\rule{\textwidth}{0.4pt}

\subsection*{Step 1. Why is there a positive margin $\delta>0$ on a sphere?}

At $x$ we know Alexandrov’s second-order expansion with Hessian $H\succ 0$:
\[
\phi(x+u) = \phi(x)+\langle \nabla \phi(x), u\rangle + \tfrac12 \langle H u,u\rangle + o(\|u\|^2).
\]

\begin{itemize}
    \item For very small $u\ne 0$, the difference
    \[
    \phi(x+u)-\phi(x)-\langle \nabla \phi(x),u\rangle
    \]
    is approximately $\tfrac12\langle H u,u\rangle$.
    \item Because $H\succ0$, this quadratic form is $\ge \tfrac{\lambda_{\min}(H)}{2}\|u\|^2$. In particular, the difference is \textbf{strictly positive} for all small nonzero $u$.
    \item The map $u\mapsto \phi(x+u)-\phi(x)-\langle \nabla \phi(x),u\rangle$ is continuous (difference of continuous functions). On the compact sphere $\{u:\|u\|=r\}$, a continuous positive function attains a \textbf{positive minimum}. Call it $\delta$.
\end{itemize}

So
\[
\delta := \min_{\|u\|=r} \big[\phi(x+u)-\phi(x)-\langle \nabla\phi(x),u\rangle \big] >0.
\]
That’s the justification for $\delta>0$.

\medskip
\noindent\rule{\textwidth}{0.4pt}

\subsection*{Step 2. Why does the same bound extend “beyond radius $r$”?}

Convexity: for any $y$ with $\|y-x\|>r$, write it as
\[
y = (1-\alpha)x + \alpha(x+u),\qquad \alpha>1,\ \|u\|=r.
\]
(Just take the radial ray through $x$; any point beyond the sphere lies at a convex combination with $\alpha>1$.)

By convexity of $\phi$,
\[
\phi(y) \ge (1-\alpha)\phi(x) + \alpha \phi(x+u).
\]
\authors{Convexity inequality is used in the wrong direction}

Subtract the affine plane at $x$:
\[
\phi(y) - \phi(x) - \langle \nabla\phi(x), y-x\rangle
\ge \alpha\big[ \phi(x+u)-\phi(x)-\langle \nabla\phi(x),u\rangle \big].
\]
But by definition of $\delta$, the bracket is $\ge \delta$. Hence the whole difference is $\ge \alpha \delta$. In particular it is at least $\delta$. So
\[
\phi(y)\ \ge\ \phi(x)+\langle \nabla\phi(x), y-x\rangle + \delta\qquad(\|y-x\|\ge r).
\]
That’s exactly the inequality denoted (BG).

\medskip
\noindent\rule{\textwidth}{0.4pt}

\subsection*{Step 3. Why is this useful?}

\begin{itemize}
    \item For $y$ away from $x$, $\phi(y)$ lies \textbf{strictly above} the supporting affine plane at $x$ by at least $\delta$.
    \item When you add the perturbation $t h(y)$, the worst that can happen is you lose $\tfrac{|t|M}{2}\|y-x\|^2$ (from the quadratic Taylor bound on $h$). But because $\|y-x\|\ge r$, this loss is at most $\tfrac{|t|M}{2}r^2$.

    \authors{The justification above is wrong because $y$ is outside the ball.}
    
    \item So as long as $|t|\le \tfrac{2\delta}{M r^2}$, the positive margin $\delta$ dominates the possible negative perturbation from $t h$. Then $\phi(y)+t h(y)$ still stays above the affine plane through $(x,f(x))$ with slope $\nabla\phi(x)+t\nabla h(x)$. That’s what ensures $a_t(y)\le f(y)$ for all $y$ outside $B(x,r)$.
\end{itemize}

\medskip
\noindent\rule{\textwidth}{0.4pt}

\textbf{Summary.}
\begin{itemize}
    \item $\delta>0$ comes from the fact that the quadratic correction in Alexandrov’s expansion is strictly positive definite, and hence the difference with the tangent plane is continuous and strictly positive on the sphere.
    \item The extension beyond radius $r$ is just convexity along rays, which scales up that positive difference.
\end{itemize}

Together this ensures that the affine minorant with slope $\nabla\phi(x)+t\nabla h(x)$ is globally below $\phi+t h$ when $|t|$ is small (depending on $x$), giving the local equality between $f$ and $f^{**}$ at $x$.

\medskip

\end{ChatGPTBox} 

\authors{Somehow we received a second answer to our last prompt.}

\begin{ChatGPTBox}       

Good — let me slow down and make the “away from $x$” part completely transparent.  

We are trying to construct, for $f=\phi+t h$, an \textbf{affine minorant}
\[
a_t(y) = f(x)+\langle p_t,\,y-x\rangle, \qquad p_t=\nabla\phi(x)+t\nabla h(x),
\]
valid for all $y$, that touches $f$ at $x$.

\medskip
\noindent\rule{\textwidth}{0.4pt}

\subsection*{Step 1. Why there is a positive “gap” at distance $r$}

At $t=0$, by Alexandrov differentiability with Hessian $H=\nabla^2\phi(x)\succ 0$, we have a \textbf{local quadratic support}:
\[
\phi(x+u)\ge \phi(x)+\langle\nabla\phi(x),u\rangle + \tfrac{m}{2}\|u\|^2,\quad \|u\|\le r
\]
for some $m>0,r>0$. In particular, if you restrict to the \emph{sphere} $\|u\|=r$, this lower bound says
\[
\phi(x+u)-\phi(x)-\langle\nabla\phi(x),u\rangle \ge \tfrac{m}{2}\,r^2.
\]

So if you define
\[
\delta := \min_{\|u\|=r}\big(\phi(x+u)-\phi(x)-\langle\nabla\phi(x),u\rangle\big),
\]
then by the inequality above $\delta \ge \tfrac{m}{2}r^2 > 0$.  
That is where the positive number $\delta$ comes from.

\medskip
\noindent\rule{\textwidth}{0.4pt}

\subsection*{Step 2. Extend the gap outside the ball}

Now why does this inequality propagate to all $\|y-x\|\ge r$? This is simply convexity of $\phi$.

Fix any direction $u\in \mathbb S^{d-1}$. Consider the function along the ray
\[
g(s) := \phi(x+s u) - \phi(x) - \langle \nabla\phi(x), s u\rangle,\quad s\ge 0.
\]

\begin{itemize}
    \item From the quadratic support near 0 we know $g(s)\ge \tfrac{m}{2}s^2$ for small $s$.
    \item Convexity of $\phi$ implies convexity of $g(s)$ as a function of $s$.
    \item In particular, $g(s)/s$ is nondecreasing in $s>0$.
\end{itemize}

So if at radius $r$ we have $g(r)\ge \delta$, then for every $s\ge r$,
\[
\frac{g(s)}{s}\ge\frac{g(r)}{r}\ge\frac{\delta}{r}.
\]

Thus for every $y=x+s u$ with $\|y-x\|=s\ge r$,
\[
\phi(y)\ge \phi(x)+\langle \nabla\phi(x),y-x\rangle + \tfrac{\delta}{r}\,\|y-x\|.
\]

This is the \textbf{“boundary gap” inequality}:

\authors{The proof of the boundary gap is now correct.}

Convexity ensures that once there is a gap $\delta$ on the sphere of radius $r$, that gap (rescaled linearly with distance) persists for all larger radii.

\medskip
\noindent\rule{\textwidth}{0.4pt}

\subsection*{Step 3. Why this helps with the affine minorant}

For $y$ with $\|y-x\|\ge r$, we thus have
\[
\phi(y) \ge \phi(x)+\langle\nabla\phi(x),y-x\rangle + \tfrac{\delta}{r}\,\|y-x\|.
\]

Adding the Taylor lower bound for $h$,
\[
h(y)\ge h(x)+\langle\nabla h(x),y-x\rangle - \tfrac{M}{2}\|y-x\|^2,
\]
we get for $f(y)=\phi(y)+t h(y)$:
\[
f(y)\ge f(x)+\langle\nabla\phi(x)+t\nabla h(x),y-x\rangle
+\frac{\delta}{r}\|y-x\| - \tfrac{|t|M}{2}\|y-x\|^2.
\]

So indeed,
\[
f(y)\ge a_t(y) + \left(\frac{\delta}{r}\|y-x\| - \tfrac{|t|M}{2}\|y-x\|^2\right).
\]

The term in parentheses is nonnegative whenever
\[
|t|\le \frac{2\delta}{M r^2}.
\]
\authors{The model is again bounding $- \tfrac{|t|M}{2}\|y-x\|^2 \ge - \tfrac{|t|M r}{2}\|y-x\|$ which is true when $\|y-x\|\le r$, whereas here $y$ is outside the ball: $\|y-x\|\ge r$.}

That gives the “away from $x$” condition.

\medskip
\noindent\rule{\textwidth}{0.4pt}

\textbf{Summary.}
\begin{itemize}
    \item The constant $\delta>0$ is guaranteed because the local quadratic lower bound forces a strict gap on the sphere $\|y-x\|=r$.
    \item Convexity propagates this gap outwards (linearly in $\|y-x\|$), so you can control all $y$ outside the ball.
    \item That’s how you obtain a \textbf{global affine minorant} touching $f$ at $x$ whenever $|t|$ is small enough.
\end{itemize}

So yes: you can phrase the whole proof in the “affine envelope” language: for each $x$ there is an affine function $a_t$ with slope $p_t=\nabla\phi(x)+t\nabla h(x)$ that underestimates $f=\phi+t h$ globally and touches it at $x$. That is exactly what shows $f^{**}(x)=f(x)$ and then $\nabla f^{**}(x)=\nabla f(x)$ at $x$.

\medskip

\medskip
\authors{[...]}

\end{ChatGPTBox} \begin{AuthorsBox}   

Outside the ball the quadratic penalty coming from $h$ becomes larger. Not $r^2$.

\medskip

\end{AuthorsBox} \begin{ChatGPTBox}       
\medskip

You’re right to be wary of the “$r^2$” that showed up before. There is a simpler, cleaner way to do the “away from $x$” part that doesn’t rely on a crude quadratic bound at large radii, and it makes explicit why the choice of $t$ can (and in general will) depend on $x$.

I’ll give a short proof that produces an affine minorant $a_t$ touching $f=\phi+t h$ at $x$ and staying below $f$ everywhere when $|t|$ is small (depending on $x$). The only inputs are:

\begin{itemize}
    \item Alexandrov at $x$: $\phi$ has a positive-definite Hessian $H=\nabla^2\phi(x)\succ0$.
    \item $h\in C^2$ (no global convexity of $h$ needed; compact support makes life even easier).
\end{itemize}

Throughout, write
\[
p_0:=\nabla\phi(x),\qquad q:=\nabla h(x).
\]

We will use the \textbf{affine-envelope characterization} of the biconjugate:
\[
f^{**}(x)=\sup \{a(x):a\text{ affine},\ a\le f\text{ on }\mathbb R^d\}.
\]
So it suffices to build \textbf{one affine minorant} $a_t$ with $a_t\le f$ everywhere and $a_t(x)=f(x)$. We’ll choose
\[
a_t(y):=f(x)+\langle p_0,\,y-x\rangle=\phi(x)+t\,h(x)+\langle \nabla\phi(x),y-x\rangle.
\]
Note: I am \textbf{not} including the $t\,q$ slope; using $p_0$ (the slope of $\phi$ at $x$) is the safe choice globally.

The task is to show:
\[
\phi(y)+t\,h(y)\ge \phi(x)+t\,h(x)+\langle p_0,\,y-x\rangle\qquad\forall y\in\mathbb R^d,
\]
for $|t|$ small (depending on $x$). We split the space into two regions.

\medskip
\noindent\rule{\textwidth}{0.4pt}

\subsection*{1) Inside a small ball $B(x,r)$}

From Alexandrov and $H\succ0$, for any small $\varepsilon\in(0,\lambda_{\min}(H)/2)$ we have (for some $r>0$)
\[
\phi(x+u)\ge \phi(x)+\langle p_0,u\rangle+\frac{m}{2}\|u\|^2,\qquad m:=\lambda_{\min}(H)-\varepsilon>0,\ \ \|u\|\le r.
\]
Also, by a Taylor bound for $h$ on $B(x,r)$,
\[
h(x+u)\ge h(x)+\langle q,u\rangle-\frac{M}{2}\|u\|^2,\qquad M:=\sup_{B(x,r)}\|\nabla^2 h\|.
\]
Therefore, for $\|u\|\le r$,
\[
\begin{aligned}
f(x+u)-a_t(x+u)
&=(\phi(x+u)-\phi(x)-\langle p_0,u\rangle) + t(h(x+u)-h(x))\\
&\ge \frac{m}{2}\|u\|^2 + t\langle q,u\rangle - \frac{|t|M}{2}\|u\|^2.
\end{aligned}
\]
To handle the linear term $t\langle q,u\rangle$, complete the square (or use Young’s inequality):
\[
t\langle q,u\rangle \ge -\,\frac{m}{4}\|u\|^2 - \frac{t^2\|q\|^2}{m}.
\]
Thus, for $\|u\|\le r$,
\[
f(x+u)-a_t(x+u)
\ge \Big(\frac{m}{2}-\frac{|t|M}{2}-\frac{m}{4}\Big)\|u\|^2 - \frac{t^2\|q\|^2}{m}.
\]
If you choose $|t|\le m/M$, the coefficient in front of $\|u\|^2$ is $\ge m/4>0$, hence the RHS is $\ge -\frac{t^2\|q\|^2}{m}$. In particular,
\[
\inf_{\|u\|\le r}\big(f(x+u)-a_t(x+u)\big)\ge -\frac{t^2\|q\|^2}{m}.
\]
Since $t\to0$, this term is $O(t^2)$. If you want \textbf{nonnegativity} (not just a tiny negative), simply shrink $r$ a bit \textbf{and} shrink $|t|$ further so that the positive quadratic part dominates the linear perturbation everywhere on $B(x,r)$; equivalently, fix $|t|\le \min \{m/M,\,{m r}/(2\|q\|) \}$. Then
\[
f(x+u)\ge a_t(x+u)\qquad(\|u\|\le r).
\]
This is the “inside the ball” part.

\emph{(If you prefer to be even more explicit: picking $|t|\le m/M$ ensures the quadratic coefficient is $\ge m/4$. Then pick $|t|\le \frac{m r}{4\|q\|}$ so that $\frac{m}{4}\|u\|^2 \ge |t|\|q\|\,\|u\|$ for all $\|u\|\le r$.)}

\authors{The condition $|t|\le \frac{m r}{4\|q\|}$ does not imply $\frac{m}{4}\|u\|^2 \ge |t|\|q\|\,\|u\|$. Same type of mistake as above.}

\medskip
\noindent\rule{\textwidth}{0.4pt}

\subsection*{2) Outside the ball $B(x,r)$}

Here we \textbf{don’t} use a quadratic bound. We use a \textbf{fixed gap} for $\phi$, and the \textbf{compactness of the support of $h$} (or, more generally, that $h$ is bounded on the annulus we consider).

First, because $\phi$ sits strictly above its tangent plane on the sphere $\partial B(x,r)$, continuity yields
\[
\delta := \min_{\|u\|=r}\big(\phi(x+u)-\phi(x)-\langle p_0,u\rangle\big)>0.
\]
Convexity along rays implies this same gap persists for any point farther out:
\[
\phi(y)\ge \phi(x)+\langle p_0, y-x\rangle + \delta,\qquad \|y-x\|\ge r.
\]

Now split the outside region into two parts:

\begin{itemize}
    \item \textbf{Far outside, where $h(y)=0$} (this is where compact support helps):  
    Choose $R_x>r$ so that $\mathrm{supp}(h)\subset B(x,R_x)$ (this is always possible for a fixed $x$, since $\mathrm{supp}(h)$ is compact).  
    For $\|y-x\|\ge R_x$, $h(y)=0$, and hence
    \[
    f(y)=\phi(y)
    \ge \phi(x)+\langle p_0, y-x\rangle
    = a_t(y) - t\,h(x),
    \]
    which is \textbf{even stronger} than needed; in fact $f(y)\ge a_t(y)$ because the extra $t\,h(x)$ appears on the other side of $a_t$’s definition (we included $t h(x)$ in $a_t$ already).
    
    \authors{The displayed inequality is not necessarily stronger.}
    \item \textbf{The annulus $r\le \|y-x\|\le R_x$}:  
    On this compact set, $h$ is bounded, so
    \[
    c_x := \sup_{r\le \|y-x\|\le R_x}\ |h(y)-h(x)|<\infty.
    \]
    Using the gap for $\phi$,
    \[
    f(y)-a_t(y)
    =\underbrace{\phi(y)-\phi(x)-\langle p_0,y-x\rangle}_{\ge \delta}
    + t\big(h(y)-h(x)\big)
    \ge \delta - |t|\,c_x.
    \]
    
    \authors{But, $x$ is not on the annulus to bound $t\big(h(y)-h(x)\big) \ge -|t|c_x$.}
    
    Hence on the annulus we ensure $f(y)\ge a_t(y)$ by choosing
    \[
    |t|\le \frac{\delta}{c_x}.
    \]
\end{itemize}

Putting the two parts together, for
\[
|t| \le t_x := \min\left( \frac{m}{M}, \frac{m r}{2\|q\|}, \frac{\delta}{c_x} \right)
\]
we have $f(y)\ge a_t(y)$ for \textbf{every} $y\in\mathbb R^d$, with equality at $x$. Thus $a_t$ is a \textbf{global affine minorant touching at $x$}, and by the affine-envelope characterization,
\[
(\phi+th)^{**}(x)\ge a_t(x)=f(x).
\]
Since $(\phi+th)^{**}\le f$ always, we conclude
\[
(\phi+th)^{**}(x)=f(x)=\phi(x)+t\,h(x).
\]

Finally, the \textbf{gradient equality at $x$} follows from the “sandwich” subgradient argument that does \emph{not} use convexity of $f$: if $g=f^{**}\le f$ and $g(x)=f(x)$, then any $p\in\partial g(x)$ satisfies the global support inequality $f(z)\ge f(x)+\langle p,z-x\rangle$ for all $z$; differentiability of $f$ at $x$ forces $p=\nabla f(x)$. Hence
\[
\nabla(\phi+th)^{**}(x)=\nabla(\phi+th)(x)=\nabla\phi(x)+t\,\nabla h(x).
\]

\medskip
\noindent\rule{\textwidth}{0.4pt}

\subsection*{Why this avoids your concern}

\begin{itemize}
    \item We \textbf{do not} try to control a negative quadratic outside the ball.  
    Instead, we use a \textbf{fixed linear gap $\delta$} for $\phi$ outside $B(x,r)$, and the fact that on the \textbf{bounded annulus} where $h$ can be nonzero we have a \textbf{uniform bound} $|h(y)-h(x)|\le c_x$. So we require $|t|\le \delta/c_x$ there.  
    For the region \textbf{beyond the annulus}, $h(y)=0$ and convexity of $\phi$ alone gives $f(y)\ge a_t(y)$.
    \item The constants $(r,m,\delta,R_x,c_x)$ depend on $x$, so $t_x$ does too. That’s expected and harmless: we only need the \textbf{pointwise} first-order information at each $x$ (to identify the Lagrangian velocity); no uniform $t$-window is required for the rest of the proof, which is carried out after OU smoothing.
\end{itemize}

If you’d like, I can compact this into a lemma you can paste—“Affine minorant touching at $x$ for $\phi+th$”—with the constants spelled out exactly as above.

\medskip

\end{ChatGPTBox} \begin{AuthorsBox}   Inside the ball, you have to take $t$ depending on $u$. Not allowed.

\medskip

\end{AuthorsBox} \begin{ChatGPTBox}       

You’re right — the way I wrote the “inside the ball” part looked like the choice of $t$ might depend on $\|u\|$. Here’s the clean fix with a \textbf{single} $t$ (depending on $x$, but \textbf{not} on $u$) that works \textbf{uniformly for all $\|u\|\le r$}.

\medskip
\noindent\rule{\textwidth}{0.4pt}

\subsection*{Goal (inside $B(x,r)$)}

We want
\[
f(x+u)-a_t(x+u)\ge 0 \qquad\text{for all }\|u\|\le r,
\]
with
\[
f=\phi+t\,h,\quad a_t(y)=f(x)+\langle p_0,\,y-x\rangle,\quad p_0:=\nabla\phi(x),\quad q:=\nabla h(x).
\]

From the one-point quadratic lower support for $\phi$ and the quadratic Taylor bound for $h$, we have (for some $m>0$, all $\|u\|\le r$)
\[
\begin{aligned}
f(x+u)-a_t(x+u)
&=(\phi(x+u)-\phi(x)-\langle p_0,u\rangle)+t(h(x+u)-h(x))\\
&\ge \frac{m}{2}\|u\|^2 + t\langle q,u\rangle - \frac{|t|M}{2}\,\|u\|^2\\
&=: \alpha\,\|u\|^2 - \beta\,\|u\|,
\end{aligned}
\]
where
\[
\alpha=\frac{m-|t|M}{2},\qquad \beta=|t|\,\|q\|.
\]
So the problem reduces to ensuring
\[
\alpha s^2-\beta s\ge 0\quad\text{for all } s\in[0,r].
\tag{Q}
\]

\authors{This approach does not sound right from the beginning, because the slope of $a_t$ is $p_0$ and not $p_0 + tq$. We see in the next section that the model is making a number of mistakes trying to justify its approach based on (Q).}

\medskip
\noindent\rule{\textwidth}{0.4pt}

\subsection*{A uniform choice of $t$ that guarantees (Q)}

\begin{itemize}
    \item First, we need $\alpha>0$, i.e.
    \[
    |t|<\frac{m}{M}.
    \tag{C1}
    \]
    \item The convex quadratic $s\mapsto \alpha s^2-\beta s$ attains its minimum on $[0,r]$ either at the endpoint $s=r$ or at the (unconstrained) minimizer $s^\star=\beta/(2\alpha)$ if $s^\star\le r$. 
    To avoid the interior minimum (which would give a negative value unless $\beta=0$), we \textbf{force} the minimizer to lie \textbf{beyond $r$}:
    \[
    \frac{\beta}{2\alpha}\ge r
    \ \Longleftrightarrow\
    \alpha r \ge \beta.
    \tag{C2}
    \]
    Under (C2), the minimum on $[0,r]$ occurs at $s=r$, and then
    \[
    \min_{s\in[0,r]}(\alpha s^2-\beta s)=\alpha r^2-\beta r\ge 0
    \quad\text{iff }\ \alpha r\ge \beta,
    \]
    i.e. the same condition.
\end{itemize}

Plugging $\alpha,\beta$ in (C2) gives the \textbf{uniform} bound
\[
\frac{m-|t|M}{2}\,r \ge |t|\,\|q\|
\ \Longleftrightarrow\
|t|\le \frac{m r}{M r+2\|q\|}.
\tag{C2'}
\]

Thus, a single choice
\[
\boxed{|t|\le t_x^{\rm in}:=\min\!\left\{\frac{m}{M},\ \frac{m r}{M r+2\|q\|}\right\}}
\]
ensures $f(x+u)-a_t(x+u)\ge 0$ \textbf{for every} $\|u\|\le r$.

\emph{(If $q=\nabla h(x)=0$, condition (C2') is trivial; you only need $|t|<m/M$ to keep $\alpha>0$.)}

\medskip
\noindent\rule{\textwidth}{0.4pt}

\subsection*{Outside the ball (quick recap without any $r^2$ issue)}

\begin{itemize}
    \item On $\partial B(x,r)$, by continuity and the strict positivity coming from the Alexandrov bound, the \textbf{gap}
    \[
    \delta:=\min_{\|u\|=r}\,\big(\phi(x+u)-\phi(x)-\langle p_0,u\rangle\big)\ >\ 0.
    \]
    Convexity along rays propagates this: for all $\|y-x\|\ge r$, $\phi(y)\ge \phi(x)+\langle p_0,y-x\rangle+\delta$.

    \item Since $h$ is continuous and the annulus $r\le\|y-x\|\le R_x$ is compact (choose $R_x$ so $\mathrm{supp}(h)\subset B(x,R_x)$), set
    \[
    c_x:=\sup_{r\le\|y-x\|\le R_x}\,|h(y)-h(x)|<\infty.
    \]
    Then
    \[
    f(y)-a_t(y)=\underbrace{\phi(y)-\phi(x)-\langle p_0,y-x\rangle}_{\ge \delta} + t\,(h(y)-h(x))
    \ge \delta - |t|\,c_x.
    \]
    
    \authors{The model is again confused on who is $x$. It is not on the annulus. However, our final proof uses a similar bound to control $h(y)-h(x)$.}
    
    So for $|t|\le \delta/c_x$, we have $f\ge a_t$ on the annulus; and beyond $R_x$ we have $h=0$, hence $f=\phi\ge a_t$ by convexity of $\phi$.
\end{itemize}

\medskip
\noindent\rule{\textwidth}{0.4pt}

\subsection*{Final uniform $t$ (depending on $x$, not on $u$)}

Take
\[
\boxed{t_x:=\min\!\left\{\,t_x^{\rm in},\ \frac{\delta}{c_x}\,\right\}=
\min\!\left\{\frac{m}{M},\ \frac{m r}{M r+2\|q\|},\ \frac{\delta}{c_x}\right\}.}
\]
Then for \textbf{every} $y\in\mathbb R^d$, $f(y)\ge a_t(y)$ when $|t|\le t_x$, with equality at $x$. Hence $a_t$ is a \textbf{global affine minorant touching at $x$}, so by the affine-envelope characterization
\[
(\phi+t h)^{**}(x)=f^{**}(x)=f(x)=\phi(x)+t\,h(x).
\]
Finally, the “sandwich” argument with differentiable $f$ at $x$ gives
\[
\nabla(\phi+th)^{**}(x)=\nabla(\phi+th)(x)=\nabla\phi(x)+t\nabla h(x).
\]

\medskip
\noindent\rule{\textwidth}{0.4pt}

\textbf{Takeaway.} The inside-ball nonnegativity can be enforced with a \textbf{single, uniform} inequality (C2'), so $t$ does \textbf{not} depend on $\|u\|$. The outside-ball part uses a fixed linear gap for $\phi$ and boundedness of $h$ on a compact annulus, giving another uniform bound $|t|\le \delta/c_x$. Combining the two yields a \textbf{pointwise} $t_x>0$ (depending on $x$, as expected) which makes the affine minorant work globally.

\medskip

\end{ChatGPTBox}

\authors{At this stage of the chat, we had enough ingredients to write a full and correct proof, see Section~\ref{sec:math}.}

\end{document}